\documentclass[12pt,reqno]{elsarticle}

\usepackage{amsfonts}
\usepackage{amsthm}
\usepackage{amsmath}
\usepackage{amssymb}
\usepackage{mathtools} %for prescript
\usepackage{xcolor}
\usepackage{soul} %for strikethrough
\usepackage{mathrsfs}

\usepackage[margin=0.9in]{geometry}
\usepackage{amsmath,amssymb,amsthm,graphicx,amsxtra, setspace}
\usepackage[utf8]{inputenc}
\usepackage{mathrsfs}
\usepackage{eucal}
\usepackage{hyperref}
\usepackage{upgreek}
\usepackage{mathtools}
\usepackage{multirow}
\usepackage{mathabx}
\usepackage{graphicx,type1cm,eso-pic,color}
\allowdisplaybreaks

\usepackage[pagewise]{lineno}

\usepackage{graphicx,eurosym}
\usepackage{hyperref}
\usepackage{mathtools}

\usepackage[cyr]{aeguill}

\colorlet{darkblue}{blue!50!black}

\hypersetup{
	colorlinks,%
	citecolor=blue,%
	filecolor=red,%
	linkcolor=darkblue,%
	urlcolor=blue,%
	pdfnewwindow=true,%
	pdfstartview={FitH}
}

\colorlet{darkblue}{red!100!black}

\newtheorem{theorem}{Theorem}[section]
\newtheorem{lemma}[theorem]{Lemma}
\newtheorem{proposition}[theorem]{Proposition}

\newtheorem{definition}[theorem]{Definition}

\newtheorem{remark}[theorem]{Remark}

\newtheorem{hypothesis}[theorem]{Hypothesis}

\let\originalleft\left
\let\originalright\right
\renewcommand{\left}{\mathopen{}\mathclose\bgroup\originalleft}
\renewcommand{\right}{\aftergroup\egroup\originalright}

\renewcommand{\d}{\/\mathrm{d}\/}

\def\w{\textbf{W}^{\varepsilon}_{{\theta}^{\varepsilon}}}

\def\L{\mathbb{L}}
\def\A{\mathrm{A}}

\def\C{\mathrm{C}}
\def\f{\boldsymbol{f}}

\def\B{\mathrm{B}}
\def\D{\mathrm{D}}
\def\y{\boldsymbol{y}}

\def\x{\boldsymbol{x}}

\def\v{\boldsymbol{v}}
\def\w{\boldsymbol{w}}
\def\W{\mathrm{W}}

\def\N{\mathbb{N}}

\def\V{\mathbb{V}}

\def\P{\mathrm{P}}
\def\u{\boldsymbol{u}}
\def\H{\mathbb{H}}
\def\n{\boldsymbol{n}}

\newcommand{\R}{\mathbb{R}}

\renewcommand{\d}{\/\mathrm{d}\/}

\begin{document}
	%	\linenumbers

		\begin{frontmatter}
		
		\title{Theory of weak asymptotic autonomy of pullback stochastic weak attractors and its applications to 2D stochastic Euler equations driven by multiplicative noise}
		
		%%this line removes the date, but space is still left for it;
		%if used, remove the \vspace{-1cm}
		\date{}
		
		%this gives the date in the form Mon 30 Jan 2012, 8:57pm;
		%if used, retain the \vspace{-1cm}
		%\date{\shortdayofweekname{\day}{\month}{\year}{ }\mydate\today}

		%\author[]{}
		\author[1]{Kush Kinra\fnref{fn1}}
		\ead{kkinra@ma.iitr.ac.in}
		\fntext[fn1]{Funding: CSIR, India, File No. 09/143(0938)/2019-EMR-I}
		
	%	\and 
		
		\author[2]{Manil T. Mohan\corref{cor1}\fnref{fn2}}
		\ead{maniltmohan@gmail.com,maniltmohan@ma.iitr.ac.in}
		\cortext[cor1]{Corresponding author}
		\address{$^1$Department of Mathematics, Indian Institute of Technology Roorkee-IIT Roorkee,
			Haridwar Highway, Roorkee, Uttarakhand 247667, India}
		\fntext[fn2]{Funding: DST-INSPIRE: IFA17-MA110}
		
		% Latex won't make the title unless told:
		%\maketitle
		
		%%to remove the space left for date, use:

		\begin{abstract}
			\noindent The two dimensional stochastic Euler equations (EE) perturbed by a linear multiplicative noise of It\^o type on the bounded domain $\mathcal{O}$ have been considered in this work. Our first aim is to prove the existence of \textsl{global weak (analytic) solutions} for stochastic EE when the divergence free initial data  $\u^*\in\H^1(\mathcal{O})$, and the external forcing  $\f\in\mathrm{L}^2_{\mathrm{loc}}(\R;\mathbb{H}^1(\mathcal{O}))$. In order to prove the existence of weak solutions, a vanishing viscosity technique has been adopted. In addition, if $\mathrm{curl}\ \u^*\in\mathrm{L}^{\infty}(\mathcal{O})$ and $\mathrm{curl}\ \f\in\mathrm{L}^{\infty}_{\mathrm{loc}}(\R;\mathrm{L}^{\infty}(\mathcal{O}))$, we establish that the global weak (analytic) solution is unique. This work appears to be the first one to discuss the existence and uniqueness of global weak (analytic) solutions for stochastic EE driven by linear multiplicative noise. Secondly, we prove the existence of a \textsl{pullback stochastic weak attractor} for stochastic \textsl{non-autonomous}  EE using the abstract theory available in the literature. Finally, we propose an abstract theory for \textsl{weak asymptotic autonomy} of pullback stochastic weak attractors. Then we consider the 2D stochastic EE perturbed by a linear multiplicative noise as an example to discuss that how to prove the weak asymptotic autonomy for concrete stochastic partial differential equations.
			As EE  do not contain any dissipative term, the results on attractors (deterministic and stochastic) are available in the literature for dissipative (or damped) EE only. Since we are considering stochastic EE without dissipation, all the results of this work for 2D stochastic EE perturbed by a linear multiplicative noise are totally new. 
		\end{abstract}
	%	\maketitle 
	\end{frontmatter}

	\section{Introduction} \label{sec1}\setcounter{equation}{0}
	The motion of gases and liquids are described by fluid dynamics. If the material density of the  fluid is a constant (that is, the divergence free condition is satisfied by the fluid velocity), then we call it as an  incompressible fluid. The Motion of a viscous incompressible fluid flow is described by the well-known Navier-Stokes equations. In the case of inviscid flow (zero viscosity), the motion of the ideal incompressible fluid in a bounded  domain $\mathcal{O}\subset\R^2$ is described by the Euler equations (EE) \begin{equation}\label{Euler2D}
		\left\{
		\begin{aligned}
			\frac{\partial\u}{\partial t}+(\u\cdot\nabla)\u+\nabla p &=\boldsymbol{f},\ \ \ \text{ in } \ \mathcal{O}\times(\tau,\infty), \\ \nabla\cdot\u&=0, \hspace{5mm} \text{ in } \ \mathcal{O}\times[\tau,\infty), \\	
		\end{aligned}
		\right.
	\end{equation}
where $\u(x,t)\in \R^2$ stands for the velocity field, $p(x,t)\in\R$ denotes the pressure field and $\f(x,t)\in\R^2$ is an external forcing, for all $(x,t)\in\mathcal{O}\times(\tau,\infty)$.
\subsection{The stochastic model} 
This work aims to study the dynamics of the stochastic counterpart of the system \eqref{Euler2D} on a bounded domain. Let $\mathcal{O}\subset\R^2$ be a bounded domain with $\mathrm{C}^2$-boundary $\partial\mathcal{O}$. Given $\tau\in\R$, we consider the following non-autonomous 2D stochastic EE on $\mathcal{O}$:
\begin{equation}\label{1}
	\left\{
	\begin{aligned}
		\d \u+\left[(\u\cdot\nabla)\u+\nabla p\right]\d t&=\boldsymbol{f}\d t+ \sigma\u\d\mathrm{W}, \ \text{ in } \ \mathcal{O}\times(\tau,\infty), \\ \nabla\cdot\u&=0, \hspace{23mm} \text{ in } \ \mathcal{O}\times[\tau,\infty), \\
		\u\cdot\boldsymbol{n} &=0, \hspace{23mm} \text{ on } \ \partial\mathcal{O}\times[\tau,\infty), \\
		\u|_{t=\tau}&=\u_{\tau}, \hspace{21mm} \text{ in } \ \mathcal{O},\\			
	\end{aligned}
	\right.
\end{equation}
 where the coefficient $\sigma>0$ is known as the \emph{noise intensity} and $\boldsymbol{n}$ is the outward normal to $\partial\mathcal{O}$. Here, the stochastic integral is understood in the sense of  It\^o and $\W=\W(t,\omega)$ is a one-dimensional two-sided Wiener process defined on a filtered probability space $(\Omega, \mathscr{F}, (\mathscr{F}_{t})_{t\in\R}, \mathbb{P})$ (see Subsection \ref{EPDS} below).

\subsection{Literature survey} 
%After involving in the research of acoustics, hydrostatics and hydraulics for many years, L. Euler presented his works on fluids  in \cite{Euler}. %Due to Bernoulli and D’Alembert, Euler’s work was highlighted in the sense that the equations are concise and express in an idealized manner the nature of motion of the fluid (cf. \cite{Constantin}).
 The derivation of the Euler equations can be found in the several books and articles, for example, see \cite{CM,MB,MP}, etc. 
The author  in \cite{Kato} proved that there exists a unique classical solution of the non-stationary EE on a bounded domain of the plane. The author in \cite{Kozono} proved the existence of a unique global classical solution of 2D EE in a time-dependent domain (see \cite{HH} for existence and uniqueness of weak solutions). The global well-posedness  (existence, uniqueness and continuous dependence on the data) of the system  \eqref{Euler2D} in $\L^2(\mathcal{O})$ and $\mathbb{H}^1(\mathcal{O})$ is not known till  date. Several mathematicians have taken Euler equations into consideration and established that there exists a weak solution to the system \eqref{Euler2D} with divergence free initial data $\u_{\tau}\in\mathbb{H}^1(\mathcal{O})$ (cf. \cite{Bardos1972,HBFF2000,Judovic} etc.). In addition, if the divergence free initial data $\u_{\tau}\in\mathbb{H}^1(\mathcal{O})$ and $\mathrm{curl\ }\u_{\tau}\in\mathrm{L}^{\infty}(\mathcal{O})$, then there exists a unique weak solution to the system \eqref{Euler2D} (cf. \cite{Bardos1972,Yudovich1995,Yudovich2005} etc.). For 3D EE, we refer the interested readers to the works \cite{Elgindi,TKGP,RogerTemam}, etc., and references therein. L. Onsager in 1949 conjectured that the energy of an ideal incompressible flow will conserve if it is H\"older continuous with exponent greater than 1/3, and  is known as \emph{Onsager's conjecture}. If the domain is either torus or the whole space (that is, in the absence of physical boundary), the  Onsager's conjecture is proved in \cite{CWT}. Also, the works \cite{BT2018,BTW} deal with the Onsager's conjecture on bounded domains. %We refer readers to the paper \cite{GJS} for Onsagar's conjecture for 2D EE with random forcing.

Let us now discuss the solvability results of 2D stochastic Euler equations. The existence of solutions  for 2D stochastic EE on torus is studied in the work \cite{CC}  using nonstandard analysis. The authors in the works \cite{HB1999,HBFF1999,Brzezniak,Kim} proved the existence of weak/martingale/strong solutions (in probabilistic sense) for  2D stochastic EE on bounded domains. In addition, the uniqueness of solutions of 2D stochastic EE driven by additive noise is also shown in \cite{HBFF1999,Kim}. The local and global existence of smooth solutions for  stochastic EE driven by multiplicative noise have been demonstrated in \cite{HV}. Moreover, in \cite{HV}, the authors considered  the stochastic EE with a linear multiplicative noise of It\^o type separately and proved the existence of a global smooth solution. Some authors (for instance,  see \cite{JSR,KF,MTMSSS}, etc.) have considered 2D as well as 3D stochastic EE driven by L\'evy noise and proved the existence of local smooth solutions. For 3D stochastic EE, we refer the interested readers to the works \cite{BM,HZZ}, etc., also. Recently, the existence of invariant measures is established for 2D stochastic dissipative EE driven by additive noise on a torus  in \cite{HBBF2020}. For 2D stochastic EE driven by a linear multiplicative noise,  the existence of invariant measures will be addressed in a future work.

 The concepts of global (pullback) attractors of deterministic autonomous (non-autonomous) dynamical systems (DS) can be found in the  books \cite{CLR,CV,Robinson,R.Temam}, etc. The notion of attractors of deterministic DS was generalized to global/pullback random attractors for autonomous/non-autonomous compact/non-compact stochastic/random DS in the articles \cite{BCF,CDF,CF,SandN_Wang}, etc. The dynamics of deterministic dissipative EE can be found in the works \cite{CV1,CVZ,VVC,CIZ,ZSC}, etc.  Later, these abstract theories have been extensively applied in the articles \cite{KM3,PeriodicWang,rwang2,XC}, etc. to several physically relevant stochastic models.

 \subsection{Motivations}
 In order to apply the classical theory of global/pullback attractors (deterministic or random), one needs to define a semigroup (or evolution process) on a metric space $(\mathrm{X}, \d_{\mathrm{X}})$ which is continuous on $\mathrm{X}$. In this work, we are considering a stochastic DS associated with stochastic EE perturbed by a white noise which is neither dissipative nor has regularizing effects. We show the existence of a unique global weak (analytic) solution (see Definition \ref{GWS} below) with initial data from space $\mathcal{W}=\{\v\in\V: \nabla\land\v\in\mathrm{L}^{\infty}(\mathcal{O})\}$ (see Subsection \ref{sec3.1} for functional spaces). Furthermore, the stochastic flow is neither continuous nor compact with respect to the topology $\d_{\mathcal{W}}(\u,\v)=\|\u-\v\|_{\V}+\|\nabla\land\u-\nabla\land\v\|_{\mathrm{L}^{\infty}(\mathcal{O})}$. But the stochastic flow is compact with respect to an another topology $\delta(\u,\v)=\|\u-\v\|_{\H}$ on $\mathcal{W}$ which is  weaker  than $\d_{\mathcal{W}}$ in some sense. For such kind of stochastic flows, the author in \cite{HB2000} came up with a notion of attractor with respect to a pair of topologies $(\mathrm{d}_{\mathcal{W}},\delta)$ known as \textsl{stochastic weak attractor} (SWA) where the attractor is $\mathrm{d}_{\mathcal{W}}$-bounded and $\delta$-compact (see Theorem \ref{SWA} below). As the uniqueness holds for the stochastic EE when the initial data is in $\mathcal{W}$, the stochastic flow generated by stochastic EE satisfies the evolution property. But it does not enjoys the continuity in phase space $\mathcal{W}$ which is used to prove the invariance property of attractors in the classical theory of global/pullback random attractors. Despite of continuity in phase space $\mathcal{W}$, the author in \cite{HB2000} considered an another assumption which assures  the invariance and the existence of a $\mathrm{d}_{\mathcal{W}}$-bounded and $\delta$-compact SWA (see Hypothesis \ref{H1} below). Moreover, this abstract theory was applied to prove the existence of SWA for \textsl{2D stochastic  dissipative EE driven by additive noise} in a bounded domain (\cite[Theorem 4.6]{HB2000}). We use the abstract theory developed in  \cite{HB2000} to prove the existence of a minimal SWA for  \textsl{2D stochastic Euler equations perturbed by a linear multiplicative noise} (see Theorem \ref{M-Thm1} below). 
 
 The effect of time-dependent forcing creates the non-autonomous feature of dynamical systems which could be the most important feature distinguishing from autonomous. Fundamentally, if the forcing $\f(x,t)$ (time-dependent) in \eqref{1} asymptotically approaches to a time-independent forcing $\f_{\infty}$ in some sense, then the random dynamics of non-autonomous system converges to the random dynamics of autonomous system. This phenomena is called the \textsl{asymptotic autonomy} or \textsl{asymptotically autonomous robustness of random dynamics}. From now onward, we use the terminology \textsl{pullback deterministic (stochastic) weak attractors} for non-autonomous deterministic (stochastic) dynamical systems and \textsl{deterministic (stochastic) weak attractors} for autonomous deterministic (stochastic) dynamical systems. Our ambition is to establish the asymptotically autonomous robustness of pullback stochastic weak attractors of \eqref{1} when the non-autonomous forcing $\f_{\infty}$ satisfies the following hypotheses:

 \begin{hypothesis}\label{Hypo_f-N}
 	$\f(\cdot)$ converges to $\f_{\infty}$ in the following sense:
 	\begin{align*}
 		\lim_{\tau\to +\infty}\int_{0}^{\infty}
 		\|\f(t+\tau)-\f_{\infty}\|^2_{\L^2(\mathcal{O})}\d t=0.
 	\end{align*}
 \end{hypothesis}
 \begin{hypothesis}\label{Hypo-f-3}
 	For the time-dependent external forcing $\f\in\mathrm{L}^{2}_{\emph{loc}}(\R;\L^2(\mathcal{O}))$ and $\nabla\land\f\in\mathrm{L}^{\infty}_{\mathrm{loc}}(\R;\mathrm{L}^{\infty}(\mathcal{O}))$, there exist two numbers $\delta_1,  \delta_2\in[0,\frac{\sigma^2}{2})$ such that for all $t\in\R$
 	\begin{align}\label{G3}
 		&\sup_{s\geq \tau}\int_{-\infty}^{0} \bigg[e^{\delta_1 t}\|\f(t+s)\|^2_{\L^2(\mathcal{O})}+e^{\delta_1 t}\|\nabla\land\f(t+s)\|^2_{\mathrm{L}^2(\mathcal{O})}+e^{\delta_2 t}\|\nabla\land\f(t+s)\|_{\mathrm{L}^{\infty}(\mathcal{O})}\bigg]\d t<+\infty,
 	\end{align}
 	%for all  $\kappa>0$ and $\tau\in\mathbb{R}$, 
 	where $\nabla\land \f= \mathrm{curl}\f$.
 \end{hypothesis}

 \begin{remark}
 1. 	An example of Hypotheses \ref{Hypo_f-N} and \ref{Hypo-f-3} is
 	$\f(x,t)=\f_\infty(x)e^t+\f_\infty(x)$ with $\f_{\infty}\in\L^2(\mathcal{O})$ and $\nabla\land\f_{\infty}\in\mathrm{L}^{\infty}(\mathcal{O})$.
 	
 	2. Hypothesis \ref{Hypo_f-N} implies that $\sup\limits_{s\geq \tau}\int_{-\infty}^{0} e^{k t}\|\f(t+s)\|^2_{\L^2(\mathcal{O})}\d t<+\infty$ for all $k>0$ and $\tau\in\R$ (cf. \cite[Lemma 4.1]{RWYL}). Therefore the first term in the integral is finite for $\delta_1\in(0,\frac{\sigma^2}{2})$ under Hypothesis \ref{Hypo_f-N}. 
 \end{remark}

The concept of asymptotically autonomous robustness of pullback random attractors has been considered in several articles, see  \cite{KRM,LX1,RKM,WL,XL,ZL}, etc. for 2D Navier-Stokes equations, 2D MHD equations, $g$-Navier-Stokes equations and Brinkman-Forchheimer equations etc. As we have already discussed, the classical theory for pullback random attractors depends on the continuity of the flow map in the phase space $\mathcal{W}$ (cf. \cite{SandN_Wang}). But the flow map of the system \eqref{1} is not continuous in $\mathcal{W}$. Motivated from the work \cite{HB2000}, we introduce a notion of asymptotic autonomy of pullback stochastic weak attractors relative to a pair of $(\mathrm{d}_{\mathcal{W}},\delta)$ topologies or \textsl{weak asymptotic autonomy of pullback stochastic weak attractors}. By the weak asymptotic autonomy of pullback stochastic weak attractors, we means that the minimal pullback stochastic weak attractors $\{\mathscr{A}(t,\omega)\}_{t\in\R,\omega\in\Omega}$ ($\d_{\mathcal{W}}$-bounded and $\delta$-compact) converges towards the minimal stochastic weak attractor $\{\mathscr{A}_{\infty}(\omega)\}_{\omega\in\Omega}$ ($\d_{\mathcal{W}}$-bounded and $\delta$-compact) in the following sense:
 	\begin{align}
 	\lim_{t\to+\infty}\mathrm{dist}_{\mathcal{W}}^{\delta}(\mathscr{A}(t,\omega),\mathscr{A}_{\infty}(\omega))=0, \ \ \mathbb{P}\text{-a.e. } \omega\in\Omega,
 \end{align}
 where, $\mathrm{dist}_{\mathcal{W}}^{\delta}(\cdot,\cdot)$ denotes the Hausdorff semi-distance between two non-empty subsets of the Banach space $\mathcal{W}$ with respect to the metric $\delta$, that is, for non-empty sets $A,B\subset \mathcal{W}$ $$\mathrm{dist}_{\mathcal{W}}^{\delta}(A,B)=\sup_{a\in A}\inf_{b\in B} \delta(a,b).$$

We prove that three sufficient conditions guarantee the weak asymptotic autonomy of pullback stochastic weak attractors, see  Theorem \ref{SWA} below. Then we apply this abstract theory to 2D stochastic EE driven by a linear multiplicative noise to show how to check the sufficient conditions for concrete stochastic partial differential equations.

\subsection{Novelties} 
To the best of our knowledge, there are no works available in the literature for the random dynamics of 2D stochastic EE.  Most of the works available in the literature  related to the random dynamics and  asymptotic analysis are for 2D deterministic/stochastic dissipative EE,  cf. \cite{HB2000,HBFF2000,HBBF2020,VVC,CIZ,ZSC} etc. Since we are not considering any kind of damping term in stochastic EE, the  linear multiplicative structure of the It\^o type noise coefficient helps us to prove the results of this work.  This work appears to be the first one concerning the random dynamics of  2D stochastic EE perturbed by a linear multiplicative noise (both It\^o and Stratonovich noise).  The major aims and novelties of this work are:
\begin{itemize}
	\item [(i)] \textsl{Existence and uniqueness of the global weak (analytic) solution of stochastic Euler system \eqref{1} by using a vanishing viscosity method and compactness arguments.}
	
	\item [(ii)] \textsl{Existence of a minimal pullback stochastic weak attractor for the non-autonomous  and a minimal stochastic weak attractor for the autonomous  stochastic Euler system \eqref{1}. We achieve this goal by using the abstract theory developed in \cite{HB2000}.}
	
	\item [(iii)] \textsl{Weak asymptotic autonomy of pullback stochastic weak attractors of the stochastic Euler system \eqref{1}. In order to obtain this, we first establish an abstract theory of weak asymptotic autonomy of pullback stochastic weak attractors in Section \ref{sec2} and then apply this to the system \eqref{1}.}
\end{itemize}

\subsection{Difficulties and approaches}
Since Euler equations do not have any regularizing term, the existence of global weak (analytic) solutions has been proved by using a  vanishing viscosity method. Firstly, we convert the stochastic Euler system \eqref{1} into an equivalent pathwise deterministic system (see \eqref{EuEq} below) with the help of an stationary change of variable (see \eqref{COV} below). Next, we consider a Navier-Stokes type system (see system \eqref{NSE} below) and obtain the existence of global weak (analytic) solutions to the Euler system \eqref{1} as viscosity coefficient approaches to zero. We prove the existence of global weak (analytic) solutions of the stochastic Euler system \eqref{1} when the divergence free initial data $\u_{\tau}\in\H^1(\mathcal{O})$, and external forcing $\f\in\mathrm{L}^2_{\mathrm{loc}}(\R;\L^2(\mathcal{O}))$ and $\nabla\land\f\in\mathrm{L}^2_{\mathrm{loc}}(\R;\mathrm{L}^2(\mathcal{O}))$. In order to obtain the uniqueness of weak solutions, we need more regular divergence free initial data and external forcing. We establish the existence of a unique global weak (analytic) solution when divergence free initial data $\u_{\tau}\in\H^1(\mathcal{O})$, $\nabla\land\u_{\tau}\in\mathrm{L}^{\infty}(\mathcal{O})$, $\f\in\mathrm{L}^2_{\mathrm{loc}}(\R;\L^2(\mathcal{O}))$ and $\nabla\land\f\in\mathrm{L}^{\infty}_{\mathrm{loc}}(\R;\mathrm{L}^{\infty}(\mathcal{O}))$. 

For the existence of pullback stochastic weak attractors, we borrowed the abstract theory introduced in \cite{HB2000}. The linear structure of It\^o type multiplicative  noise   plays a crucial role to  show the existence of  $\d_{\mathcal{W}}$-bounded and $\delta$-compact absorbing set. On contrast to the works available in the literature \cite{HBFF2000,HBBF2020},  the presence of linear multiplicative noise helped us to prove the results for the system \eqref{1} without a linear damping term.  Since the weak asymptotic autonomy concept is not available in the literature, we develop an  abstract theory  in Section \ref{sec2} and then  apply it to show  the weak asymptotic autonomy of pullback stochastic weak attractors for the system \eqref{1}.

\subsection{Outline}
The rest of the sections are organized as follows: In the next section, we propose an abstract theory  for weak asymptotic autonomy of pullback stochastic weak attractors. In Section \ref{sec3}, we prove the existence of a unique global weak (analytic) solution of the stochastic Euler system \eqref{1}. The existence of minimal pullback stochastic weak attractors for the stochastic Euler system \eqref{1} is established in Section \ref{sec4}. In light of the abstract theory developed in Section \ref{sec2}, we demonstrate the weak asymptotic autonomy of pullback stochastic weak attractors in the final section.

\section{Abstract theory}\label{sec2}\setcounter{equation}{0}

The abstract theory of weak global attractors for autonomous dynamical systems was developed in \cite{HBFF2000}. Later, the author in \cite{HB2000} extended this theory to non-autonomous deterministic (stochastic) dynamical systems named as \textsl{deterministic (stochastic) weak attractors}. Let us first recall the abstract theory available in the work \cite{HB2000} for pullback stochastic weak attractors. Then we develop an abstract theory for weak asymptotic autonomy of  pullback stochastic weak attractors.

\subsection{Pullback stochastic weak attractors and stochastic weak attractors}\label{AT}

Let $(\mathcal{W},\d_{\mathcal{W}})$ be a complete metric space, $\delta$ be an another metric on $\mathcal{W}$ and $(\Omega,\mathscr{F},\mathbb{P})$ be a probability space.  Let $\{S(t,\tau,\omega)\}_{t\geq \tau, \omega\in\Omega}:\mathcal{W}\to\mathcal{W}$ be a family of mappings which satisfies the following hypothesis:
\begin{hypothesis}\label{H1}
	Let $\{S(t,\tau,\omega)\}_{t\geq \tau,\  \omega\in\Omega}:\mathcal{W}\to\mathcal{W}$ be a family of mappings, satisfying for $\mathbb{P}$-a.e. $\omega\in\Omega,$ the following two properties:
	\begin{itemize}
		\item [(i)] If a sequence $\{\w_n\}_{n\in\N}$ is $\d_{\mathcal{W}}$-bounded and $\delta$-convergent to $\w$ in $\mathcal{W},$ then $S(t,\tau,\omega)\w_n$ is $\delta$-convergent to $S(t,\tau,\omega)\w$, for all $t\geq \tau$.
		\item [(ii)] \emph{Evolution property}:
		\begin{align}
			S(t,r,\omega)S(r,\tau,\omega)\w=S(t,\tau,\omega)\w,  \text{ for all }  \tau\leq r\leq t\text{ and for all } \w\in\mathcal{W}.
		\end{align}
	\end{itemize}
\end{hypothesis}

Let $\mathscr{B}(\mathcal{W})$ be the $\sigma$-algebra of the metric space $\mathcal{W}$ and $\mathcal{H}$ be a separable metric space such that $\mathcal{W}\subset\mathcal{H}$ with continuous embedding. The metric $\delta$ is the metric endowed on $\mathcal{H}$ and $\mathcal{W}$ is closed in $\mathcal{H}$. Let us assume that the following hypothesis holds:

\begin{hypothesis}\label{H2}
	For all $t\in\R$ and $\w\in\mathcal{W}$, the mapping $(\tau,\omega)\mapsto S(t,\tau,\omega)\w$ is measurable from $((-\infty,t]\times\Omega,\mathscr{B}((-\infty,t])\times\mathscr{F})$ to $(\mathcal{W},\mathscr{B}(\mathcal{H})\cap\mathcal{W})$.
\end{hypothesis}

\begin{definition}
	For given $t\in\R$ and $\omega\in\Omega$, a bi-parametric set $\mathcal{B}(t,\omega)$ is called \emph{$\d_{\mathcal{W}}$-bounded absorbing set} at time $t$ if for all $\d_{\mathcal{W}}$-bounded set $B\subset\mathcal{W}$, there exists a time $\tau_0(B)$ depending  on $B$ such that 
	\begin{align*}
		S(t,\tau,\omega)B\subset\mathcal{B}(t,\omega),  \ \ \text{ for all } \ \tau\leq \tau_0.
	\end{align*}
\end{definition}

\begin{definition}
	A family of mappings $\{S(t,\tau,\omega)\}_{t\geq \tau,  \omega\in\Omega}$ is called \emph{asymptotically $\d_{\mathcal{W}}/\delta$ -compact} if we can find a measurable set $\widetilde{\Omega}\subset\Omega$ with $\mathbb{P}(\widetilde{\Omega})=1$ such that for all $t\in\R$ and $\omega\in\widetilde{\Omega}$, there exists a $\d_{\mathcal{W}}$-absorbing set $\mathcal{B}(t,\omega)$ with $\delta$-compactness. 
\end{definition}

The following theorem provides the existence of a measurable pullback stochastic weak attractor proved in \cite{HB2000}.
\begin{theorem}[{\cite[Theorem 3.2]{HB2000}}]\label{SWA}
	Let  Hypotheses \ref{H1} and \ref{H2} hold. Assume that for each $t\in\R$, there exists a measurable set $\Omega_{t}\subset\Omega$ with $\mathbb{P}(\Omega_{t})=1$, such that for all $\omega\in \Omega_t$, there exists a $\d_{\mathcal{W}}$-bounded and $\delta$-compact absorbing set. Then for $\mathbb{P}$-a.e. $\omega\in\Omega_t$, the set $$\mathscr{A}(t,\omega)=\overline{\bigcup_{B\subset\mathcal{W}}\Omega^{\delta}(B,t,\omega)}^{\delta},\ \ \text{ where }\ \ \Omega^{\delta}(B,t,\omega)=\bigcap_{r<t}\overline{\bigcup_{\tau<r}S(t,\tau,\omega)B}^{\delta},$$ is a measurable \textsl{pullback stochastic weak attractor}, that is,
	\begin{itemize}
		\item [(i)] $\mathscr{A}(t,\omega)$ is non-empty, it is $\d_{\mathcal{W}}$-bounded and $\delta$-compact,
		\item [(ii)] $\mathscr{A}(t,\omega)$ is invariant under the mapping $S(\cdot,\cdot,\cdot)$, that is, $$S(t,\tau,\omega)\mathscr{A}(\tau,\omega)=\mathscr{A}(t,\omega),\ \text{ for all }  \ t\geq\tau,$$
		\item [(iii)] For every $\d_{\mathcal{W}}$-bounded set $B\subset\mathcal{W}$, $$\lim_{\tau\to-\infty}\delta(S(t,\tau,\omega)B,\mathscr{A}(t,\omega))=0,\ \text{ for all }  \ t\geq\tau,$$
		\item [(iv)] $\mathscr{A}(t,\omega)$ is measurable with respect to the $\mathbb{P}$-completion of $\mathscr{F}$. 
		\item[(v)] $\mathscr{A}(t,\omega)$ is the minimal $\d_{\mathcal{W}}$-bounded and  $\delta$-closed set with property $\mathrm{(iii)}$, that is, if $\widehat{\mathscr{A}}(t,\omega)$ is an another  $\d_{\mathcal{W}}$-bounded and  $\delta$-closed set with the property $\mathrm{(iii)}$, then $\mathscr{A}(t,\omega)\subset\widehat{\mathscr{A}}(t,\omega)$.
	\end{itemize}
\end{theorem}
\begin{proof}
	Properties $\mathrm{(i)}$-$\mathrm{(iv)}$ have been proved in \cite[Theorem 3.2]{HB2000}. Let us prove the property $\mathrm{(v)}$. Note that if  $\widehat{\mathscr{A}}(t,\omega)$ satisfies property $\mathrm{(iii)}$, then we have $\Omega^{\delta}(B,t,\omega)\subset\widehat{\mathscr{A}}(t,\omega)$  for all $\d_{\mathcal{W}}$-bounded sets $B\subset\mathcal{W}$. Since $\widehat{\mathscr{A}}(t,\omega)$ is $\d_{\mathcal{W}}$-bounded and $\delta$-closed, $\mathscr{A}(t,\omega)\subset\widehat{\mathscr{A}}(t,\omega)$.
\end{proof}
Let us now discuss the abstract theory for stochastic weak attractors. 
\begin{hypothesis}\label{H3}
	There exists a group $\{\theta_{t}\}_{t\in\R}$ of measure preserving transformations of  $(\Omega,\mathscr{F},\mathbb{P})$ 
	such that 
	\begin{align*}
		T(t,\tau,\omega)\w=T(t-\tau,0,\theta_{\tau}\omega)\w,\ \ \ \mathbb{P}\text{-a.s.},
	\end{align*}
	for all $t\geq\tau$ and $\w\in\mathcal{W}$.
\end{hypothesis}

If a family of mappings $\{T(t,\tau,\omega)\}_{t\geq \tau, \omega\in\Omega}$ satisfies Hypothesis \ref{H3}, then the attractor does not depend on the time parameter. The following theorem assures the existence of a measurable \textsl{stochastic weak attractor} when Hypothesis \ref{H3} is satisfied (see \cite[Theorem 2.2]{CDF} also). 
\begin{theorem}\label{SWA1}
	Let Hypotheses \ref{H1}, \ref{H2} and \ref{H3} hold. %and that the time shift $\{\theta_t\}_{t\in\R}$ is ergodic.
	 Assume that for $\mathbb{P}$-a.e. $\omega\in \Omega$, there exists a $\d_{\mathcal{W}}$-bounded and $\delta$-compact absorbing set at time $0$. Then for $\mathbb{P}$-a.e. $\omega\in\Omega$, the set $$\mathscr{A}_{\infty}(\omega)=\overline{\bigcup_{B\subset\mathcal{W}}\Omega^{\delta}(B,\omega)}^{\delta},\ \ \text{ where }\ \ \Omega^{\delta}(B,\omega)=\bigcap_{r<0}\overline{\bigcup_{\tau<r}T(0,\tau,\omega)B}^{\delta},$$ is a measurable \textsl{stochastic weak attractor}, that is,
	\begin{itemize}
		\item [(i)] $\mathscr{A}_{\infty}(\omega)$ is non-empty, it is $\d_{\mathcal{W}}$-bounded and $\delta$-compact,
		\item [(ii)] $\mathscr{A}_{\infty}(\omega)$ is invariant under the mapping $T(\cdot,\cdot,\cdot)$, that is,  $$T(t,\tau,\omega)\mathscr{A}_{\infty}(\theta_{\tau}\omega)=\mathscr{A}_{\infty}(\theta_{t}\omega),\ \text{ for all } \ t\geq\tau,$$
		\item [(iii)] For every $\d_{\mathcal{W}}$-bounded set $B\subset\mathcal{W}$, $$\lim_{\tau\to-\infty}\delta(T(t,\tau,\omega)B,\mathscr{A}_{\infty}(\theta_{t}\omega))=0,\ \text{ for all }  \ t\geq\tau,$$
		\item [(iv)] $\mathscr{A}_{\infty}(\omega)$ is measurable with respect to the $\mathbb{P}$-completion of $\mathscr{F}$. 
		\item[(v)] $\mathscr{A}_{\infty}(\omega)$ is the minimal $\d_{\mathcal{W}}$-bounded and $\delta$-closed set with property $\mathrm{(iii)}$, that is, if $\widehat{\mathscr{A}}_{\infty}(\omega)$ is an another  $\d_{\mathcal{W}}$-bounded and $\delta$-closed set with the property $\mathrm{(iii)}$, then $\mathscr{A}_{\infty}(\omega)\subset\widehat{\mathscr{A}}_{\infty}(\omega)$.
	\end{itemize}
\end{theorem}

\subsection{Weak asymptotic autonomy}\label{AT1}

In this subsection, our aim is to establish an abstract theory for the \emph{weak asymptotic autonomy} of \textsl{pullback stochastic weak attractors}. In applications to stochastic evolution equations perturbed by white noise, the sample space $\Omega$ in the probability space $(\Omega,\mathscr{F},\mathbb{P})$ is 
\begin{align}\label{Omega-Y}
	 \Omega=\{\omega\in \mathrm{C}(\R;\mathrm{Y}):\omega(0)=0\}=:\mathrm{C}_0(\R;\mathrm{Y}),
\end{align} 
where $\mathrm{Y}$ is some Banach space. In this case, there exists a group $\{\theta_{t}\}_{t\in\R}$ of measure preserving transformations of  $(\Omega,\mathscr{F},\mathbb{P})$ which is given by (\cite{Arnold,CDF})
\begin{align}
	(\theta_{t}\omega)(s)=\omega(t+s)-\omega(t), \ \ \ \text{ for all }\ \ s,t\in\R.
\end{align}
In fact, we are proposing an abstract theory on those  probability spaces where the existence of a group $\{\theta_{t}\}_{t\in\R}$ of measure preserving transformations  is known.
\begin{remark}
	Note that if the existence of a group $\{\theta_{t}\}_{t\in\R}$ of measure preserving transformations is known, it implies the existence of an $\Omega_{t}$  in Theorem \ref{SWA}. In this case, one can consider $\Omega_{t}$ as  the set $\theta_{t}\Omega$.  
\end{remark}

Let $\{\mathscr{A}(t,\omega)\}_{t\in\R,\omega\in\Omega}$ and $\mathcal{B}(t,\omega)$ be a \textsl{minimal pullback stochastic weak attractor} and \textsl{$\d_{\mathcal{W}}$-bounded and $\delta$-compact absorbing set} at time $t$, respectively, associated with the non-autonomous stochastic dynamical system $\{S(t,\tau,\omega)\}_{t\geq \tau, \omega\in\Omega}$ on $\mathcal{W}$. Also, let $\{\mathscr{A}_{\infty}(\omega)\}_{\omega\in\Omega}$ and $\mathcal{B}_{\infty}(\omega)$ be a \textsl{minimal stochastic weak attractor} and \textsl{$\d_{\mathcal{W}}$-bounded and $\delta$-compact absorbing set},  respectively, associated with the autonomous stochastic dynamical system $\{T(t,\tau,\omega)\}_{t\geq \tau, \omega\in\Omega}$ on $\mathcal{W}$. We show that $\{\mathscr{A}(t,\omega)\}_{t\in\R,\omega\in\Omega}$  upper semicontinuously forward converges to $\{\mathscr{A}_{\infty}(\omega)\}_{\omega\in\Omega}$ with respect to the metric $\delta$.

\begin{remark}
	For the sample space $\Omega=\mathrm{C}_0(\R;\mathrm{Y})$, as the existence of a group $\{\theta_{t}\}_{t\in\R}$ of measure preserving transformations of $(\Omega,\mathscr{F},\mathbb{P})$ is known, it is remarkable that the autonomous stochastic dynamical system $\{T(t,\tau,\omega)\}_{t\geq \tau, \omega\in\Omega}$ over $(\Omega,\mathscr{F},\mathbb{P})$ satisfies  Hypothesis \ref{H3}.
\end{remark}
The following theorem provides the abstract result for the weak asymptotic autonomy of pullback stochastic weak attractors.

\begin{theorem}\label{WAA-MT}
	Let $\{\mathscr{A}(\tau,\omega)\}_{\tau\in\R,\omega\in\Omega}$ and $\{\mathscr{A}_{\infty}(\omega)\}_{\omega\in\Omega}$ are $\d_{\mathcal{W}}$-bounded and $\delta$-compact. Suppose that the following properties hold.
	\begin{itemize}
		\item [(i)] $\overline{\bigcup\limits_{t\geq\tau}\mathscr{A}(t,\omega)}^{\delta}$ is $\d_{\mathcal{W}}$-bounded and $\delta$-compact.
		\item [(ii)]The mappings $S(\cdot,\cdot,\cdot)$ and $T(\cdot,\cdot,\cdot)$ satisfy,
		\begin{align}\label{Con}
			\lim\limits_{\tau\to+\infty}\delta(S(t+\tau,\tau,\omega)\w_{\tau}, T(t+\tau,\tau,\omega)\w_{0})=0,  \ \text{ for all } t\geq0 \text{ and } \omega\in\widetilde{\Omega},
		\end{align}
		whenever $\{\w_{\tau}\}_{\tau\in\R}$ is a $d_{\mathcal{W}}$-bounded sequence, $\w_{0}\in\mathcal{W}$ and $\lim\limits_{\tau\to+\infty}\delta(\w_{\tau}, \w_{0})=0,$ where $\widetilde{\Omega}\subset\Omega$ is a $\theta_{t}$-invariant full measure set. 
		\item [(iii)]There exists a $t_0>0$ such that $\mathcal{B}_{t_0}(\omega):=\overline{\bigcup\limits_{t\geq t_0}\mathcal{B}(t, \omega)}^{\delta}$ is $\d_{\mathcal{W}}$-bounded and $\delta$-compact set. 
	\end{itemize}
	Then, the \textbf{weak asymptotic autonomy} follows for $\mathscr{A}$, that is,
	\begin{align}\label{WAA}
		\lim_{t\to+\infty}\mathrm{dist}_{\mathcal{W}}^{\delta}(\mathscr{A}(t,\omega),\mathscr{A}_{\infty}(\omega))=0, \ \ \mathbb{P}\text{-a.e. } \omega\in\Omega,
	\end{align}
	where, $\mathrm{dist}_{\mathcal{W}}^{\delta}(\cdot,\cdot)$ denotes the Hausdorff semi-distance between two non-empty subsets of the Banach space $\mathcal{W}$ with respect to the metric $\delta$, that is, for non-empty sets $A,B\subset \mathcal{W}$ $$\mathrm{dist}_{\mathcal{W}}^{\delta}(A,B)=\sup_{a\in A}\inf_{b\in B} \delta(a,b).$$
\end{theorem}
\begin{proof}
	In order to prove the convergence given in \eqref{WAA}, it is enough to show that $\mathbb{P}(\Omega_0)=1$, where $$\Omega_0=\left\{\omega\in\Omega:\lim\limits_{t\to+\infty}\mathrm{dist}_{\mathcal{W}}^{\delta}\left(\mathscr{A}(t,\omega),\mathscr{A}_{\infty}(\omega)\right)=0\right\}.$$ Let us assume that $\mathbb{P}(\Omega_{0})<1,$ then $\mathbb{P}(\Omega\backslash\Omega_{0})>0$. Let $\Omega_1=(\Omega\backslash\Omega_{0})\cap\widetilde{\Omega}$, where $\widetilde{\Omega}$ is the $\theta_{t}$-invariant full measure set given in \eqref{Con}. Then $\mathbb{P}(\Omega_1)>0$ and $\theta_{s}\Omega_1\subset\widetilde{\Omega}$, for all $s\in\R$. 
	
	Let $\omega\in\Omega_1$ be fixed. Since $\omega\notin\Omega_{0}$, we can find $\varepsilon>0$ and $t_n\geq 2 t_0>0$ (WLOG) with $t_n\uparrow\infty$ such that 
	\begin{align*}
		\mathrm{dist}_{\mathcal{W}}^{\delta}\left(\mathscr{A}(t_n,\omega),\mathscr{A}_{\infty}(\omega)\right)\geq 3\varepsilon,\ \ \text{ for all } \ n\in\N.
	\end{align*}
	
	In view of $\d_{\mathcal{W}}$-boundedness and $\delta$-compactness of $\mathscr{A}(t_n,\omega)$, there exists $\w_n\in\mathscr{A}(t_n,\omega)$ such that 
	\begin{align}\label{C1}
		\delta\left(\w_n,\mathscr{A}_{\infty}(\omega)\right)=\mathrm{dist}_{\mathcal{W}}^{\delta}\left(\mathscr{A}(t_n,\omega),\mathscr{A}_{\infty}(\omega)\right)\geq 3\varepsilon,\ \ \text{ for all } n\in\N.
	\end{align}
	
	By the invariance property of pullback stochastic weak attractors, we also have that there exists $\x_n\in\mathscr{A}\left(\frac{t_n}{2},\theta_{-\frac{t_n}{2}}\omega\right)$ such that $S(t_n,\frac{t_n}{2},\omega)\x_n=\w_n.$
	%\begin{align}
	%	S(t_n,\frac{t_n}{2},\omega)\x_n=\w_n.
	%\end{align}
	%Since $t_n\uparrow\infty$, we can find $n_0\in\N$ such that $t_{n_0}\geq t_0>0$. 
	Now by condition $(\mathrm{i})$, there exists an element $\x\in \overline{\bigcup\limits_{t_n\geq 2t_0}\mathscr{A}\left(\frac{t_n}{2},\theta_{-\frac{t_n}{2}}\omega\right)}^{\delta}$ with 
	\begin{align}\label{C2}
		\lim\limits_{n\to+\infty}\delta(\x_n,\x)=0.
	\end{align}
	
	From \eqref{C2} and condition $(\mathrm{ii})$, we get
	\begin{align*}
		\lim_{n\to+\infty}\delta\left(S\left(t_n,\frac{t_n}{2},\omega\right)\x_{n}, T\left(t_n,\frac{t_n}{2},\omega\right)\x\right)=0,
	\end{align*}
	which confirms the existence of an $N_1\in\N$ such that 
	\begin{align}\label{C3}
		\delta\left(S\left(t_n,\frac{t_n}{2},\omega\right)\x_{n}, T\left(t_n,\frac{t_n}{2},\omega\right)\x\right)\leq\varepsilon, \ \ \ \text{ for all } \ n\geq N_1.	
	\end{align}
	
	By the invariance property of $\mathscr{A}$ and the absorption property of absorbing set $\mathcal{B}$, there exists a $\tau<0$ such that 
	\begin{align*}
		\mathscr{A}\left(\frac{t_n}{2},\theta_{-\frac{t_n}{2}}\omega\right)\ =\ S\left(\frac{t_n}{2},\tau,\omega\right)\mathscr{A}\left(\tau,\theta_{\tau-t_n}\omega\right)\ \subset\ \mathcal{B}\left(\frac{t_n}{2}, \theta_{-\frac{t_n}{2}}\omega\right),
	\end{align*}  
	which gives 
	\begin{align}\label{C4}
		%\mathscr{A}_{t_0}(\omega):=
		\x\in	\overline{\bigcup\limits_{t_n\geq 2t_0}\mathscr{A}\left(\frac{t_n}{2},\theta_{-\frac{t_n}{2}}\omega\right)}^{\delta}\ \subset\ \overline{\bigcup\limits_{t_n\geq 2t_0}\mathcal{B}\left(\frac{t_n}{2}, \theta_{-\frac{t_n}{2}}\omega\right)}^{\delta}=\mathcal{B}_{t_0}\left(\theta_{-\frac{t_n}{2}}\omega\right).
	\end{align}  
	
	By condition $(\mathrm{iii})$, \eqref{C4} and the attracting property of $\mathscr{A}_{\infty}$, there exists $N_2\in\N$ such that 
	\begin{align}\label{C5}
		\delta\left(T\left(t_n,\frac{t_n}{2},\omega\right)\x, \mathscr{A}_{\infty}(\omega)\right)\leq\varepsilon,  \ \ \ \text{ for all } \ n\geq N_2.	
	\end{align}
	
	Combining \eqref{C3} and \eqref{C5}, we deduce for all $n\geq N:=\max\{N_1,N_2\}$ that 
	\begin{align*}
		&\mathrm{dist}_{\mathcal{W}}^{\delta}\left(\mathscr{A}(t_n,\omega),\mathscr{A}_{\infty}(\omega)\right)\nonumber\\&=\delta\left(\w_n,\mathscr{A}_{\infty}(\omega)\right)\nonumber\\&\leq \delta\left(S\left(t_n,\frac{t_n}{2},\omega\right)\x_{n}, T\left(t_n,\frac{t_n}{2},\omega\right)\x\right)+\delta\left(T\left(t_n,\frac{t_n}{2},\omega\right)\x, \mathscr{A}_{\infty}(\omega)\right)\nonumber\\&\leq 2\varepsilon,
	\end{align*}
	which contradicts \eqref{C1}.
\end{proof}

	\section{Global Solvability of 2D stochastic Euler equations driven by multiplicative noise}\label{sec3}\setcounter{equation}{0}
	We first provide the necessary function spaces needed for the further analysis. Then, we define a transformation which helps to convert our stochastic system into an equivalent pathwise deterministic system. Finally, we prove the existence and uniqueness of the 2D stochastic Euler system \eqref{1}. 
	\subsection{Function spaces} \label{sec3.1}
	For $p\in[1,\infty]$, the scalar as well as vector valued  Lebesgue spaces will be denoted by $\mathrm{L}^p(\mathcal{O})=\mathrm{L}^p(\mathcal{O};\mathbb{R})$ and  $\L^p(\mathcal{O})=\mathrm{L}^p(\mathcal{O};\mathbb{R}^2)$, respectively. %For $p\in[1,\infty]$, the scalar as well as vector valued  Sobolev spaces will be denoted by $\mathrm{W}^{1,p}(\mathcal{O})=\mathrm{W}^{1,p}(\mathcal{O};\mathbb{R})$ and  $\mathbb{W}^{1,p}(\mathcal{O})=\mathrm{W}^{1,p}(\mathcal{O};\mathbb{R}^2)$, respectively. 
	For $s>0$ and $p\in[1,\infty]$,  the usual scalar and vector valued Sobolev spaces will be represented by $\mathrm{W}^{s,p}(\mathcal{O})=\mathrm{W}^{s,p}(\mathcal{O};\mathbb{R})$ and $\mathbb{W}^{s,p}(\mathcal{O})=\mathrm{W}^{s,p}(\mathcal{O};\mathbb{R}^2)$,  respectively.  For $p=2$, the Hilbertian  Sobolev spaces will be denoted by $\mathrm{H}^{s}(\mathcal{O})=\mathrm{W}^{s,2}(\mathcal{O})$ and $\mathbb{H}^{s}(\mathcal{O})=\mathbb{W}^{s,2}(\mathcal{O})$.  %Let $\mathcal{O}$ be a open connected bounded subset of $\R^2$ with $\mathrm{C}^2$-boundary $\partial\mathcal{O}$. 
	%We denote $\C_0^{\infty}(\mathcal{O};\R^2)$ for the space of all infinitely differentiable functions  ($\R^2$-valued) with compact support in $\mathcal{O}$. 
	Let us define 
	\begin{align*} 
		%	\mathcal{V}&:=\{\u\in\C_0^{\infty}(\mathcal{O},\R^2):\nabla\cdot\u=0\},\\
		\mathbb{H}&:=\{\u\in\L^2(\mathcal{O}):\nabla\cdot\u=0 \ \text{ and }\ \u\cdot\boldsymbol{n}\big|_{\partial\mathcal{O}}=0\},\\
		\mathbb{V}&:=\{\u\in\H^1(\mathcal{O}):\nabla\cdot\u=0  \ \text{ and }\  \u\cdot\boldsymbol{n}\big|_{\partial\mathcal{O}}=0\},
	\end{align*}
	where $\boldsymbol{n}$ is the unit outward normal to the boundary $\partial\mathcal{O}$, and $\u\cdot\n\big|_{\partial\mathcal{O}}$ should be understood in the sense of trace in $\H^{-1/2}(\partial\mathcal{O})$ (cf. \cite[Chapter 1, Theorem 1.2]{Temam}).
	The spaces $\H$ and $\V$ are endowed with the norms
	\begin{align*}
		\|\u\|_{\H}^2=\int_{\mathcal{O}}|\u(x)|^2\d x\ \text{ and }\ \|\u\|_{\V}^2=\int_{\mathcal{O}}\left[|\u(x)|^2+|\nabla\u(x)|^2\right]\d x,
	\end{align*}
	respectively. Let $(\cdot,\cdot)$ denote the inner product in the Hilbert space $\H$ and $\langle \cdot,\cdot\rangle $ represent the induced duality between the spaces $\V$  and its dual $\V'$. Note that $\H$ can be identified with its dual $\H'$. Furthermore, we have the Gelfand triple $\V\hookrightarrow\H \cong\H'\hookrightarrow\V'$ with dense and continuous embedding, and the embedding $\V\hookrightarrow\H$ is compact.

	\subsection{Equivalent pathwise deterministic system}\label{EPDS}
	In the system \eqref{1}, $\W(t,\omega)$ is the standard scalar Wiener process on the probability space $(\Omega, \mathscr{F}, \mathbb{P}),$ where $$\Omega=\{\omega\in \mathrm{C}(\R;\R):\omega(0)=0\}=:\C_0(\R;\R),$$ endowed with the compact-open topology given by the complete metric
	\begin{align*}
		d_{\Omega}(\omega,\omega'):=\sum_{m=1}^{\infty} \frac{1}{2^m}\frac{\|\omega-\omega'\|_{m}}{1+\|\omega-\omega'\|_{m}},\ \text{ where }\  \|\omega-\omega'\|_{m}:=\sup_{-m\leq t\leq m} |\omega(t)-\omega'(t)|,
	\end{align*}
	and $\mathscr{F}$ is the Borel sigma-algebra induced by the compact-open topology of $(\Omega,d_{\Omega}),$ $\mathbb{P}$ is the two-sided Wiener measure on $(\Omega,\mathscr{F})$. Also, define $\{\theta_{t}\}_{t\in\R}$ by 
	\begin{equation*}
		\theta_{t}\omega(\cdot)=\omega(\cdot+t) -\omega(t), \ \ \   t\in\R\ \text{ and }\ \omega\in\Omega.
	\end{equation*}
	Hence, $(\Omega,\mathscr{F},\mathbb{P},\{\theta_{t}\}_{t\in\R})$ is a metric dynamical system.

	Let us now consider
	\begin{align}\label{OU1}
		y(\theta_{t}\omega) =  \int_{-\infty}^{t} e^{-(t-\xi)}\d \W(\xi), \ \ \omega\in \Omega,
	\end{align} which is the stationary solution of the one dimensional Ornstein-Uhlenbeck equation
	\begin{align}\label{OU2}
		\d y(\theta_t\omega) +  y(\theta_t\omega)\d t =\d\W(t).
	\end{align}
	It is known from \cite{FAN} that there exists a $\theta$-invariant subset $\widetilde{\Omega}\subset\Omega$ of full measure such that $y(\theta_t\omega)$ is continuous in $t$ for every $\omega\in \widetilde{\Omega},$ and
	\begin{align}
		%	\mathbb{E}\left(|y(\theta_s\omega)|^{\xi}\right)&=\frac{\Gamma\left(\frac{1+\xi}{2}\right)}{\sqrt{\pi\sigma^{\xi}}}, \ \text{ for all } \xi>0, s\in \R,\label{Z2}\\
		\lim_{t\to \pm \infty} \frac{\omega(t)}{|t|}=\lim_{t\to \pm \infty} \frac{|y(\theta_t\omega)|}{|t|}&=	\lim_{t\to \pm \infty} \frac{1}{t} \int_{0}^{t} y(\theta_{\xi}\omega)\d\xi =\lim_{t\to \infty} e^{-\delta t}|y(\theta_{-t}\omega)| =0,\label{Z3}
	\end{align}
	for all $\delta>0$, where $\Gamma$ is the Gamma function. For further analysis of this work, we do not distinguish between $\widetilde{\Omega}$ and $\Omega$. 
	
	Define a new variable $\v$ by 
	\begin{align}\label{COV}
		\v(t;\tau,\omega,\v_{\tau})=e^{-\sigma y(\theta_{t}\omega)}\u(t;\tau,\omega,\u_{\tau}) \ \ \
		\text{ with }
		\ \ \	\v_{\tau}=e^{-\sigma y(\theta_{\tau}\omega)}\u_{\tau},
	\end{align}
	where $\u(t;\tau,\omega,\u_{\tau})$ and $y(\theta_{t}\omega)$ are the solutions of \eqref{1} and \eqref{OU2}, respectively. Then $\v(\cdot):=\v(\cdot;\tau,\omega,\v_{\tau})$ satisfies:
	\begin{equation}\label{EuEq}
		\left\{
		\begin{aligned}
			\frac{\d\v}{\d t}+\left[\frac{\sigma^2}{2}-\sigma y(\theta_{t}\omega)\right]\v+e^{\sigma y(\theta_{t}\omega)}(\v\cdot\nabla)\v+e^{-\sigma y(\theta_{t}\omega)}\nabla p&=e^{-\sigma y(\theta_{t}\omega)}\f ,\hspace{5mm} \text{ in } \ \mathcal{O}\times(\tau,\infty), \\ \nabla\cdot\v&=0, \hspace{17mm} \text{ in } \ \mathcal{O}\times(\tau,\infty), \\
			\v\cdot\boldsymbol{n} &=0,\hspace{17mm} \text{ on } \ \partial\mathcal{O}\times(\tau,\infty), \\
			\v|_{t=\tau}&=\v_{\tau}, \hspace{15mm} \text{ in } \mathcal{O},
		\end{aligned}
		\right.
	\end{equation}
	in the distributional sense, where $\sigma>0$. %, $\v_{\tau}\in\V$ and $\f\in\mathrm{L}^2_{\mathrm{loc}}(\R;\V)$. Let $\mathcal{P}: \L^2(\mathcal{O}) \to\H$ denote the Helmholtz-Hodge orthogonal projection (cf.  \cite{OAL}). By taking the projection $\mathcal{P}$ on the system \eqref{EuEq}, we get that $\v(\cdot)$ satisfies:
	It can be easily seen from  \eqref{EuEq} that  in $\V'$
	\begin{equation}\label{Div-free}
		\left\{
		\begin{aligned}
			\frac{\d\v}{\d t}+\left[\frac{\sigma^2}{2}-\sigma y(\theta_{t}\omega)\right]\v+e^{\sigma y(\theta_{t}\omega)}(\v\cdot\nabla)\v&=e^{-\sigma y(\theta_{t}\omega)}\boldsymbol{f},  \ \ \ \ \ t>\tau, \\
			\v|_{t=\tau}&=\v_{\tau}.
		\end{aligned}
		\right.
	\end{equation}
	Next, we define the vorticity $$\varrho:=\nabla\land\v=\frac{\partial v_{2}}{\partial x_1}-\frac{\partial v_{1}}{\partial x_2}.$$ Then the vorticity equation corresponding to \eqref{EuEq} is given by
	\begin{equation}\label{vorticity}
		\left\{
		\begin{aligned}
			\frac{\d\varrho}{\d t}+\left[\frac{\sigma^2}{2}-\sigma y(\theta_{t}\omega)\right]\varrho+e^{\sigma y(\theta_{t}\omega)}(\v\cdot\nabla\varrho)&=e^{-\sigma y(\theta_{t}\omega)}\nabla\land\boldsymbol{f},  \ \ \ \ \ t>\tau, \\
			\varrho|_{t=\tau}&=\nabla\land\v_{\tau},
		\end{aligned}
		\right.
	\end{equation}
	in $\mathrm{L}^2(\mathcal{O})$, since for the 2D flows the vortex stretching term is zero.
	
	\begin{remark}
		One can use an another transformation  to change the system \eqref{1} into an equivalent pathwise deterministic  system.  For a given $t\in\R$ and $\omega\in \Omega$, let $\hat{z}(t,\omega)=e^{-\sigma\omega(t)}$. Then, $\hat{z}$ satisfies the equation 
		\begin{align}\label{Trans1}
			\d \hat{z}=\frac{\sigma^2}{2} \hat{z} \d t-\sigma \hat{z}\d \W.
		\end{align}
		Let $\v$ be a new random variable given by 
		\begin{align}\label{COV2}
			\v(t;\tau,\omega,\v_{\tau})= \hat{z}(t,\omega)\u(t;\tau,\omega,\u_{\tau}) \ \ \
			\text{ with }
			\ \ \	\v_{\tau}= \hat{z}(\tau,\omega)\u_{\tau},
		\end{align}
		where $\u(t;\tau,\omega,\u_{\tau})$ and $ \hat{z}(t,\omega)$ are the solutions of \eqref{1} and \eqref{Trans1}, respectively. Then the new random variable $\v(\cdot;\tau,\omega,\v_{\tau})$ satisfies the following system:
		\begin{equation}\label{CEE1}
			\left\{
			\begin{aligned}
				\frac{\d\v}{\d t} +\frac{\sigma^2}{2}\v+\frac{1}{\hat{z}(t,\omega)}\B\big(\v\big)+\hat{z}(t,\omega)\nabla p&= \hat{z}(t,\omega)\f , \quad t> \tau, \\ 
				\v|_{t=\tau}&=\v_{\tau}, \qquad \quad x\in \mathcal{O},
			\end{aligned}
			\right.
		\end{equation}
		which is a pathwise deterministic system and is equivalent to the system \eqref{1}. Even though  the system \eqref{CEE1} is easier to handle than  the system \eqref{EuEq}, there is a technical difficulty in using the transformation  \eqref{COV2} in the context of attractors. Since $\omega(t)$ is identified with $\mathrm{W}(t,\omega)$, the change of variable given by \eqref{COV2} involves the Wiener process explicitly. Therefore, $\hat{z}(t,\omega)$ is not a stationary process and the change of variable given by \eqref{COV2} is not stationary. This fact creates troubles if one uses the change of variable given by \eqref{COV2} in order to have conjugated random dynamical systems. Since $z(t,\omega)=e^{-\sigma y(\theta_{t}\omega)}$ is stationary, the change of variable is given by a homeomorphism (also called conjugation) which transforms one random dynamical system into an another equivalent one (cf. \cite[Proposition B.9]{GLS2}). For this reason, it is appropriate to use the change of variable mentioned in \eqref{COV}.  We also point out here that the authors in \cite{HV} used the  transformation given in \eqref{Trans1} to study the global (or local) existence of smooth pathwise solutions of 2D (or 3D) Euler equations. But they are not investigating  any qualitative properties of the solutions (cf. \cite{MRRZXZ} also). The authors in \cite{KKMTM} used the transformation given in \eqref{COV} to study the random dynamics of 2D SNSE on the whole space. 
	\end{remark}

Here, the solvability of stochastic Euler system \eqref{1} has been proved by a vanishing viscosity method. Let us consider the following \textsl{Navier-Stokes type equations}:
\begin{equation}\label{NSE}
	\left\{
	\begin{aligned}
		\frac{\d\v_{\nu}}{\d t}-\nu\Delta\v_{\nu}+\left[\frac{\sigma^2}{2}-\sigma y(\theta_{t}\omega)\right]\v_{\nu}&+e^{\sigma y(\theta_{t}\omega)}(\v_{\nu}\cdot\nabla)\v_{\nu}\\+e^{-\sigma y(\theta_{t}\omega)}\nabla p_{\nu}&=e^{-\sigma y(\theta_{t}\omega)}\f , \text{ in } \ \mathcal{O}\times(\tau,\infty), \\ \nabla\cdot\v_{\nu}&=0, \ \ \ \text{ in } \ \mathcal{O}\times(\tau,\infty), \\
		\nabla\land\v_{\nu} &=0,\ \ \ \text{ on } \ \partial\mathcal{O}\times(\tau,\infty),\\
		\v_{\nu}\cdot\boldsymbol{n} &=0,\ \ \ \text{ on } \ \partial\mathcal{O}\times(\tau,\infty), \\
		\v_{\nu}|_{t=\tau}&=\v_{\tau},\ \   \text{ in } \mathcal{O},
	\end{aligned}
	\right.
\end{equation}
in the distributional sense, where $\sigma>0$, $0<\nu\leq\nu_0$ with $\nu_0$ sufficiently small (will be specified later) and 
$\nabla\land\v_{\nu}=\frac{\partial v_{\nu,2}}{\partial x_1}-\frac{\partial v_{\nu,1}}{\partial x_2}.$

\subsection{Bilinear form and linear operator}\label{LO}
%Let us define the Stokes operator
%\begin{equation*}
%	\A_{\nu,\sigma}\v_{\nu}:=-\mathcal{P}\Delta\v_{\nu},\ \v_{\nu}\in\D(\A_{\nu,\sigma}),
%\end{equation*}where 
Let the space $\H_{\mathrm{div}}(\mathcal{O})$ be defined by  $$\H_{\mathrm{div}}(\mathcal{O}):=\{\u_{\nu}\in\L^2(\mathcal{O}):\nabla\cdot\u_{\nu}\in\mathrm{L}^2(\mathcal{O})\},$$  endowed with the inner product $(\u_{\nu},\v_{\nu})_{\mathbb{H}_{\mathrm{div}}}=(\u_{\nu},\v_{\nu})+(\mathrm{div\ }\u_{\nu},\mathrm{div\ }\v_{\nu})$. Let us  define the bilinear form $a(\cdot,\cdot):\V\times\V\to\R$ by
\begin{align}\label{boundaryA}
	a(\u_{\nu},\v_{\nu})&:=\nu\int_{\mathcal{O}}\nabla\u_{\nu}(x):\nabla\v_{\nu}(x)\d x-\nu\int_{\partial\mathcal{O}}\kappa(\rho)\gamma_0(\u_{\nu}(\rho))\cdot\gamma_0(\v_{\nu}(\rho))\d \rho\nonumber\\&\qquad+\frac{\sigma^2}{2}\int_{\mathcal{O}}\u_{\nu}(x)\cdot\v_{\nu}(x)\d x, 
\end{align}
for all $\u_{\nu},\v_{\nu}\in\V$, where $\kappa(\cdot)\in\mathrm{L}^{\infty}(\partial\mathcal{O})$ is the curvature and $\gamma_0\in\mathcal{L}(\H_{\mathrm{div}}(\mathcal{O});\mathrm{H}^{-\frac{1}{2}}(\partial\mathcal{O}))\cap\mathcal{L}(\H^s(\mathcal{O});\mathrm{H}^{s-\frac{1}{2}}(\partial\mathcal{O})),$ for $s>0$ is the trace operator (cf. \cite{Bardos1972,HBFF2000,JLL-EDO}). We claim that the bilinear form \eqref{boundaryA} is well defined. Indeed, consider 
\begin{align}\label{boundaryA1}
	|a(\u_{\nu},\v_{\nu})|&\leq \nu\|\nabla\u_{\nu}\|_{\H}\|\nabla\v_{\nu}\|_{\H}+C\nu\|\gamma_0(\u_{\nu})\|_{\mathrm{H}^{-\frac{1}{2}}(\partial\mathcal{O})}\|\gamma_0(\v_{\nu})\|_{\mathrm{H}^{\frac{1}{2}}(\partial\mathcal{O})}+\frac{\sigma^2}{2}\|\u_{\nu}\|_{\H}\|\v_{\nu}\|_{\H}\nonumber\\&\leq \max\left\{\nu,\frac{\sigma^2}{2}\right\}\|\u_{\nu}\|_{\V}\|\v_{\nu}\|_{\V}+C\nu\|\u_{\nu}\|_{\H}\|\v_{\nu}\|_{\V},
	%	\nonumber\\&\leq \|\u_{\nu}\|_{\V}\|\v_{\nu}\|_{\V}+C\|\u_{\nu}\|^{\frac{1}{2}}_{\H}\|\u_{\nu}\|^{\frac{1}{2}}_{\V}\|\v_{\nu}\|_{\V},
\end{align}
which is finite for all $\u_{\nu},\v_{\nu}\in\V$, where we have used the Trace theorems \cite[Theorem 1.2, p.7]{Temam} and \cite[Theorem 3, p.206]{DMDZ}. Moreover, by the definition of $a(\cdot,\cdot)$, we have 
\begin{align}\label{3.4}
	a(\u_{\nu},\u_{\nu})&=\nu\int_{\mathcal{O}}|\nabla\u_{\nu}(x)|^2\d x+\frac{\sigma^2}{2}\int_{\mathcal{O}}|\u_{\nu}(x)|^2\d x-\nu\int_{\partial\mathcal{O}}\kappa(\rho)|\gamma_0(\u_{\nu}(\rho))|^2\d\rho.
\end{align}
By using the Trace theorem, and interpolation,  H\"older's and Young's inequalities, we obtain 
\begin{align}\label{3.5}
	\left|\int_{\partial\mathcal{O}}\kappa(\rho)|\gamma_0(\u_{\nu}(\rho))|^2\d\rho\right|&\leq C\|\gamma_0(\u_{\nu})\|_{\mathrm{L}^2(\partial\mathcal{O})}^2\leq C\|\u_{\nu}\|_{\H^{\frac{1}{2}}(\mathcal{O})}^2\leq C\|\u_{\nu}\|_{\V}\|\u_{\nu}\|_{\H}\nonumber\\&\leq \frac{1}{2}\|\u_{\nu}\|_{\V}^2+C_{\frac{1}{2}}\|\u_{\nu}\|_{\H}^2. 
\end{align}
Therefore, from \eqref{3.4}, we immediately have 
\begin{align}
	a(\u_{\nu},\u_{\nu})\geq\frac{\min\left\{\nu,\sigma^2\right\}}{2}\|\u_{\nu}\|_{\V}^2-C\nu\|\u_{\nu}\|_{\H}^2,
\end{align}
which implies that $a(\cdot,\cdot)$ is \textsf{$\V$-coercive} (cf. \cite[page 352]{GGr}).  We set  $$\D(\A_{\nu,\sigma})=\big\{\v_{\nu}\in\V\cap\H^2(\mathcal{O}): (\nabla\land\v_{\nu})|_{\partial\mathcal{O}}=0\big\},$$ and define the linear operator $\A_{\nu,\sigma}:\D(\A_{\nu,\sigma})\to \H$, as 
\begin{align}\label{3.7}\langle\A_{\nu,\sigma}\u_{\nu},\v_{\nu}\rangle=a(\u_{\nu},\v_{\nu}),\ \text{ for all }\ \u_{\nu},\v_{\nu}\in\V.\end{align}
There is also an extension $\A_{\nu,\sigma} : \V\to \V'$ of the operator $\A_{\nu,\sigma}\u_{\nu}=-\nu\Delta\u_{\nu}+\frac{\sigma^2}{2}\u_{\nu}$  by an application of the Lax-Milgram lemma.

\subsection{Trilinear form and bilinear operator}\label{BO}
Let us define the \emph{trilinear form} $b(\cdot,\cdot,\cdot):\V\times\V\times\V\to\R$ by $$b(\u,\v,\w)=\int_{\mathcal{O}}(\u(x)\cdot\nabla)\v(x)\cdot\w(x)\d x=\sum_{i,j=1}^2\int_{\mathcal{O}}\u_i(x)
\frac{\partial \v_j(x)}{\partial x_i}\w_j(x)\d x.$$ If $\u, \v$ are such that the linear map $b(\u, \v, \cdot) $ is continuous on $\V$, the corresponding element of $\V'$ is denoted by $\B(\u, \v)$. We also denote $\B(\v) = \B(\v, \v)=(\v\cdot\nabla)\v$.	An integration by parts implies (cf. \cite{Temam1})
\begin{equation}\label{b0}
	\left\{
	\begin{aligned}
		\langle\B(\u,\v),\v\rangle &= 0,\ \text{ for all }\ \u,\v \in\V,\\
		\langle\B(\u,\v),\w\rangle &=  -\langle\B(\u,\w),\v\rangle,\ \text{ for all }\ \u,\v,\w\in \V.
	\end{aligned}
	\right.\end{equation}

\subsection{Abstract formulation and existence}\label{Existence}
We can rewrite the system \eqref{1} in the following abstract form in $\V'$:
\begin{equation}\label{SNSE}
	\left\{
	\begin{aligned}
		\frac{\d\v_{\nu}}{\d t}+\A_{\nu,\sigma}\v_{\nu}+e^{\sigma y(\theta_{t}\omega)}\B(\v_{\nu})&=e^{-\sigma y(\theta_{t}\omega)}\f+\sigma y(\theta_{t}\omega)\v_{\nu},  \ \ \ \ \ t>\tau, \\
		\v_{\nu}|_{t=\tau}&=\v_{\tau},
		\end{aligned}
	\right.
\end{equation}
 where $\sigma>0$, $0<\nu\leq\nu_0$.

Let us now show that there exists a global in time solution to the system \eqref{1} in the following sense:
\begin{definition}\label{GWS}
	A stochastic process $\u(t):=\u(t;\tau,\omega,\u_{\tau})$ is a $\emph{global weak (analytic) solution}$ of the system \eqref{1} for $t\geq\tau$ and $\omega\in\Omega$ if for $\mathbb{P}$-a.e. $\omega\in\Omega$
	\begin{align*}
		\u(\cdot)\in\mathrm{C}([\tau,\infty);\H)\cap\mathrm{L}^{\infty}_{\mathrm{loc}}(\tau,\infty;\V)
	\end{align*}
and 
\begin{align*}
	(\u(t),\boldsymbol{\psi})+\int_{\tau}^{t}b(\u(\xi),\u(\xi),\boldsymbol{\psi})\d\xi=(\u_{\tau},\boldsymbol{\psi})+\int_{\tau}^{t}(\boldsymbol{f}(\xi),\boldsymbol{\psi})\d\xi +\sigma\int_{\tau}^{t}(\u(\xi),\boldsymbol{\psi})\d\W(\xi),
\end{align*}
for every $t\geq\tau$ and for every $\boldsymbol{\psi}\in\V$.
\end{definition}
 In order to prove the existence of a solution to the system \eqref{1}, it is enough to show that there exists a solution to the system \eqref{EuEq}.
\begin{theorem}\label{Existence-Eu}
	Suppose that $\u_{\tau}\in\V$, $\f\in\mathrm{L}^2_{\mathrm{loc}}(\R;\L^2(\mathcal{O}))$ and $\nabla\land\f\in\mathrm{L}^2_{\mathrm{loc}}(\R;\mathrm{L}^2(\mathcal{O}))$. Then, there exists at least one global weak (analytic) solution $\u$ of the system \eqref{1}. In addition, $\u(\cdot)$ satisfies
	\begin{align}\label{AE}
		\|\u(t)\|^2_{\V}&\leq C e^{2\sigma y(\theta_{t}\omega)}\bigg[\|\u_{\tau}\|_{\V}^2e^{-2\sigma y(\theta_{\tau}\omega)}e^{\int_{\tau}^{t}\left[-\frac{\sigma^2}{2}+2\sigma y(\theta_{r}\omega)\right]\d r}\nonumber\\&\quad+\int_{\tau}^{t}e^{\int_{\xi}^{t}\left[-\frac{\sigma^2}{2}+2\sigma y(\theta_{r}\omega)\right]\d r} e^{-2\sigma y(\theta_{\xi}\omega)}\big\{\|\f(\xi)\|^2_{\L^2(\mathcal{O})}+\|\nabla\land\boldsymbol{f}(\xi)\|_{\mathrm{L}^2(\mathcal{O})}^2\big\}\d\xi\bigg],
	\end{align}
for all $t\geq\tau$.
\end{theorem}
\begin{proof}
Let $\{e_1,e_2,\ldots,e_n,\ldots\}\subset \D(\A)$ be a complete orthonormal system in  $\H$ and let $\H_n$ be the $n$-dimensional subspace of $\H$ spanned by $\{e_1,e_2,\ldots,e_n\}$. Let us denote the projection of the space $\V'$ into $\H_n$ by $\P_n$, that is, $\P_n\x =\sum\limits_{j=1}^{n}\langle \x,e_j\rangle e_j $. As we know that every element $\x\in\H$ induces a functional $\x'\in\H$ by the formula $\langle \x',\y\rangle =(\x,\y),\; \y\in\V$, then $\P_n\big|_{\H}$, the orthogonal projection of $\H$ onto $\H_n$ is given by $\P_n\x=\sum\limits_{j=1}^{n}(\x,e_j)e_j$. Clearly,  $\P_n$ is the orthogonal projection from $\H$ onto $\H_n$. Let us consider the following finite-dimensional system:
\begin{equation}\label{P-SNSE}
	\left\{
	\begin{aligned}
		\frac{\d\v_{n,\nu}}{\d t}+\A_{\nu,\sigma}\v_{n,\nu}+e^{\sigma y(\theta_{t}\omega)}\P_n\B(\v_{n,\nu})&=e^{-\sigma y(\theta_{t}\omega)}\P_n\f+\sigma y(\theta_{t}\omega)\v_{n,\nu},  \ \ \ \ \ t>\tau, \\
		\v_{n,\nu}|_{t=\tau}&=\P_n\v_{\tau},
	\end{aligned}
	\right.
\end{equation}
where
\begin{align*}
	\P_n\v_{\tau}=\sum_{i=1}^{n}(\v_{\tau},e_i)e_i.
\end{align*}
The existence of a local solution  to the finite-dimensional system \eqref{P-SNSE} follows by an application of Carath\'eodory's existence theorem. Taking the inner product by $\v_{n,\nu}(\cdot)$ in the first equation of \eqref{P-SNSE}, and making use of \eqref{boundaryA} and \eqref{b0}, we find 
\begin{align}\label{AE1}
	&\frac{1}{2}\frac{\d}{\d t}\|\v_{n,\nu}(t)\|^2_{\H}+\nu\|\nabla\v_{n,\nu}(t)\|^2_{\H}+\left[\frac{\sigma^2}{2}-\sigma y(\theta_{t}\omega)\right]\|\v_{n,\nu}(t)\|^2_{\H}\nonumber\\&\quad=\nu\int_{\partial\mathcal{O}}\kappa(\rho)|\gamma_0(\v_{n,\nu}(t,\rho))|^2\d \rho + e^{-\sigma y(\theta_{t}\omega)}(\f(t),\v_{n,\nu}(t)), 
\end{align}
for a.e. $t\geq\tau$. Similar to \eqref{3.5}, we get
\begin{align}\label{bdry1}
	\int_{\partial\mathcal{O}}\kappa(\rho)|\gamma_0(\v_{n,\nu}(\rho))|^2\d \rho\leq C\|\v_{n,\nu}\|_{\H}\|\v_{n,\nu}\|_{\V}\leq\frac{1}{2}\|\nabla\v_{n,\nu}\|^2_{\H}+C_{\frac{1}{2}}\|\v_{n,\nu}\|^2_{\H},
\end{align}
where $C_{\frac{1}{2}}>0$ is a constant. An application of the  Cauchy-Schwarz and Young's inequalities yields 
\begin{align}\label{AE2}
	e^{-\sigma y(\theta_{t}\omega)}|(\f,\v_{n,\nu})|\leq \frac{\sigma^2}{4}\|\v_{n,\nu}\|^2_{\H}+\frac{e^{-2\sigma y(\theta_{t}\omega)}}{\sigma^2}\|\f\|^2_{\L^2(\mathcal{O})}.
\end{align}
Combining \eqref{AE1}-\eqref{AE2}, we obtain (for $0<\nu\leq\nu_0$)
\begin{align}\label{AE3}
		&\frac{\d}{\d t}\|\v_{n,\nu}(t)\|^2_{\H}+\nu\|\nabla\v_{n,\nu}(t)\|^2_{\H}+\left[\frac{\sigma^2}{2}-2 \nu C_{\frac{1}{2}}-2\sigma y(\theta_{t}\omega)\right]\|\v_{n,\nu}(t)\|^2_{\H}\nonumber\\&\leq\frac{2 e^{-2\sigma y(\theta_{t}\omega)}}{\sigma^2}\|\f(t)\|^2_{\L^2(\mathcal{O})}, 
\end{align}
for a.e. $t\geq\tau$. Making use of Gronwall's inequality, we find 
\begin{align}\label{AE4}
	&\|\v_{n,\nu}(t)\|^2_{\H}\nonumber\\&\leq  \left[\|\v_{n,\nu}(\tau)\|_{\H}^2e^{\int_{\tau}^{t}\left[-C^*(\nu,\sigma)+2\sigma y(\theta_{r}\omega)\right]\d r}+\frac{2}{\sigma^2}\int_{\tau}^{t}e^{\int_{\xi}^{t}\left[-C^*(\nu,\sigma)+2\sigma y(\theta_{r}\omega)\right]\d r} z^2(\xi,\omega)\|\f(\xi)\|^2_{\L^2(\mathcal{O})}\d\xi\right],
\end{align}
for all $t\geq\tau$, $0<\nu\leq\nu_0$ and $C^*(\nu,\sigma)=\frac{\sigma^2}{2}-2\nu C_{\frac{1}{2}}>0$ (we choose $\nu_0$ such that $\frac{\sigma^2}{2}-2\nu C_{\frac{1}{2}}>0,$ for all $0<\nu\leq\nu_0$). Since $y(\cdot)$ is continuous, $\|\v_{n,\nu}(\tau)\|_{\H}\leq\|\v_{\tau}\|_{\H}$ and $\f\in\mathrm{L}^2_{\mathrm{loc}}(\R;\L^2(\mathcal{O}))$, we infer from \eqref{AE4} that 
\begin{align}\label{AE5}
	\sup_{t\in[\tau,\tau+T]}\|\v_{n,\nu}(t)\|^2_{\H}\leq C<\infty,
\end{align}
where the constant $C$ is independent of $n$. Integrating \eqref{AE3} over $(\tau,\tau+T)$ and using \eqref{AE5}, we attain $	\{\v_{n,\nu}\}_{n\in\N} \subset \mathrm{L}^{2}(\tau,\tau+T;\V).$ 

Now, for $\boldsymbol{\psi}(\cdot)\in\mathrm{L}^{2}(\tau,\tau+T;\V)$, we consider
\begin{align}\label{Derivative}
	\left|\left\langle\frac{\d \v_{n,\nu}}{\d t},\boldsymbol{\psi}\right\rangle\right|&=\left|\left\langle\A_{\nu,\sigma}\v_{n,\nu}+e^{\sigma y(\theta_{t}\omega)}\B(\v_{n,\nu})-e^{-\sigma y(\theta_{t}\omega)}\f+y(\theta_{t}\omega)\v_{n,\nu},\boldsymbol{\psi}\right\rangle\right|\nonumber\\&\leq\nu|(\nabla\v_{n,\nu},\nabla\boldsymbol{\psi})|+\left|\int_{\partial\mathcal{O}}\kappa(y)\gamma_0(\v_{n,\nu}(t,y))\cdot\gamma_0(\boldsymbol{\psi}(t,y))\d y\right|+\left[\frac{\sigma^2}{2}+\sigma|y(\theta_{t}\omega)|\right]\nonumber\\&\quad\times|(\v_{n,\nu},\boldsymbol{\psi})|+e^{\sigma y(\theta_{t}\omega)}|b(\v_{n,\nu},\v_{n,\nu},\boldsymbol{\psi})|+e^{-\sigma y(\theta_{t}\omega)}|(\boldsymbol{f},\boldsymbol{\psi})|\nonumber\\&\leq C\left[1+|y(\theta_{t}\omega)|\right]\big(\|\v_{\nu,\tau}\|_{\V}+|y(\theta_{t}\omega)|\|\v_{n,\nu}\|_{\H}+e^{\sigma y(\theta_{t}\omega)}\|\v_{n,\nu}\|_{\H}\|\v_{n,\nu}\|_{\V}\nonumber\\&\quad+e^{-\sigma y(\theta_{t}\omega)}\|\boldsymbol{f}\|_{\L^2(\mathcal{O})}\big)\|\boldsymbol{\psi}\|_{\V},
\end{align}
which immediately gives $\{\frac{\d \v_{n,\nu}}{\d t}\}_{n\in\N} \in \mathrm{L}^{2}(\tau,\tau+T;\V')$, where we have used the fact  $\v_{n,\nu}\in\mathrm{L}^{\infty}(\tau,\tau+T;\H)\cap\mathrm{L}^{2}(\tau,\tau+T;\V)$. Hence by the Aubin-Lions compactness lemma (since $\v_{n,\nu}\in\mathrm{L}^{2}(\tau,\tau+T;\V)$ and $\frac{\d \v_{n,\nu}}{\d t}\in \mathrm{L}^{2}(\tau,\tau+T;\V')$) and Banach-Alaoglu theorem (since $\v_{n,\nu}\in\mathrm{L}^{2}(\tau,\tau+T;\V)$), there exists a subsequence (still denoted by the same) $\v_{\nu}\in \mathrm{L}^{2}(\tau,\tau+T;\V) $ such that (as $n\to+\infty$)
	\begin{align}
	\v_{n,\nu}\to&\ \v_{\nu}\ \text{ in }	\ \mathrm{L}^{2}(\tau,\tau+T;\H) \ \ \text{ strongly},\label{AE6}\\
	\v_{n,\nu}\xrightharpoonup{w}&\ \v_{\nu}\ \text{ in } \ \mathrm{L}^{2}(\tau,\tau+T;\V) \ \ \text{ weakly}.\label{AE7}
\end{align}
With the help of the convergences \eqref{AE6}-\eqref{AE7} (see \cite{Temam}), one can pass the limit $n\to+\infty$ in \eqref{P-SNSE} and obtain that $\v_{\nu}$ satisfies \eqref{SNSE} in the distributional sense. Moreover, we also deduce 
\begin{align*}
	\v_{\nu}\in\mathrm{C}([\tau,\tau+T];\H)\cap\mathrm{L}^{2}(\tau,\tau+T;\V), \ \text{ for all }\ T>0.
\end{align*} 

Our next aim is to provide uniform estimates in $\nu$ for the $\V$-norm of $	\v_{\nu}$. Let us denote by $\varrho_{\nu}=\nabla\land\v_{\nu}$. Then, \eqref{NSE} gives
\begin{equation}\label{C-SNSE}
	\left\{
	\begin{aligned}
		\frac{\d\varrho_{\nu}}{\d t}-\nu\Delta\varrho_{\nu}+\frac{\sigma^2}{2}\varrho_{\nu}+e^{\sigma y(\theta_{t}\omega)}(\v_{\nu}\cdot\nabla\varrho_{\nu})&=e^{-\sigma y(\theta_{t}\omega)}\nabla\land\boldsymbol{f}+\sigma y(\theta_{t}\omega)\varrho_{\nu},  \ \ \ t>\tau, \\
		\varrho_{\nu}|_{t=\tau}&=\nabla\land\v_{\tau}.
	\end{aligned}
	\right.
\end{equation}
 Taking the inner product with $\varrho_{\nu}$ in the first equation of \eqref{C-SNSE}, we obtain
\begin{align*}
	&\frac{\d}{\d t}\|\varrho_{\nu}(t)\|_{\mathrm{L}^2(\mathcal{O})}^2 +\left[\frac{\sigma^2}{2}-2\sigma y(\theta_{t}\omega)\right]\|\varrho_{\nu}(t)\|_{\mathrm{L}^2(\mathcal{O})}^2+2\nu\|\nabla\varrho_{\nu}(t)\|_{\mathbb{L}^2}^2\nonumber\\&\leq\frac{2e^{-2\sigma y(\theta_{t}\omega)}}{\sigma^2}\|\nabla\land\boldsymbol{f}(t)\|_{\mathrm{L}^2(\mathcal{O})}^2, \  \text{ for a.e. } t\geq\tau, 
\end{align*}
which implies 
\begin{align}\label{AE8}
	&\|\varrho_{\nu}(t)\|_{\mathrm{L}^2(\mathcal{O})}^2\nonumber\\&\leq\left[\|\varrho_{\nu}(\tau)\|^2_{\mathrm{L}^2(\mathcal{O})}e^{\int_{\tau}^{t}\left[-\frac{\sigma^2}{2}+2\sigma y(\theta_{r}\omega)\right]\d r} +\frac{2}{\sigma^2}\int_{\tau}^{t}e^{\int_{\xi}^{t}\left[-\frac{\sigma^2}{2}+2\sigma y(\theta_{r}\omega)\right]\d r}e^{-2\sigma y(\theta_{\xi}\omega)}\|\nabla\land\boldsymbol{f}(\xi)\|_{\mathrm{L}^2(\mathcal{O})}^2\d\xi\right], 
\end{align}
 for all $t\geq\tau$. Since $\Delta\v_{\nu}=\nabla(\nabla\cdot\v_{\nu})-\nabla\wedge(\nabla\wedge\v_{\nu})$, divergence free condition $\nabla\cdot\v_{\nu}=0$ assures  that $\v_{\nu}$ satisfies the following elliptic problem:
\begin{equation}\label{Elleptic-E}
	\left\{
	\begin{aligned}
		\Delta\v_{\nu}&=-\nabla^{\perp}\varrho_{\nu}, \ \ \text{ on }\ \mathcal{O}, \\
		\nabla\land\v_{\nu}|_{\partial\mathcal{O}}&=0,\\
		\v_{\nu}\cdot\boldsymbol{n}|_{\partial\mathcal{O}}&=0,
	\end{aligned}
	\right.
\end{equation}
where $\nabla^{\perp}=(\frac{\partial}{\partial x_2}, \ -\frac{\partial}{\partial x_1})$. Multiplying the first equation of \eqref{Elleptic-E} with $\v_{\nu}$ and integrating over $\mathcal{O}$, we find 
\begin{align*}
	(\Delta\v_{\nu},\v_{\nu})=-(\nabla^{\perp}\varrho_{\nu},\v_{\nu}).
\end{align*}
Making use of integration by parts and performing calculations similar to \eqref{boundaryA1}, we arrive at
\begin{align}\label{AE9}
	\|\nabla\v_{\nu}\|^2_{\H}&=2\int_{\partial\mathcal{O}}\kappa(\rho)|\gamma_0(\v_{\nu}(\rho))|^2\d \rho+2\|\varrho_{\nu}\|^2_{\mathrm{L}^{2}(\mathcal{O})}-\|\nabla\v_{\nu}\|^2_{\H}\nonumber\\&\leq C\|\v_{\nu}\|_{\H}\|\v_{\nu}\|_{\V}+2\|\varrho_{\nu}\|^2_{\mathrm{L}^{2}(\mathcal{O})}-\|\nabla\v_{\nu}\|^2_{\H}\nonumber\\&\leq C(\|\v_{\nu}\|^2_{\H}+\|\varrho_{\nu}\|^2_{\mathrm{L}^{2}(\mathcal{O})}).
\end{align}
Combining \eqref{AE4} and \eqref{AE8}-\eqref{AE9}, we arrive at
\begin{align}\label{AE10}
	\|\v_{\nu}(t)\|^2_{\V}&\leq C\bigg[\|\v_{\tau}\|_{\V}^2e^{\int_{\tau}^{t}\left[-C^*(\nu,\sigma)+2\sigma y(\theta_{r}\omega)\right]\d r}\nonumber\\&\quad+\frac{2}{\sigma^2}\int_{\tau}^{t}e^{\int_{\xi}^{t}\left[-C^*(\nu,\sigma)+2\sigma y(\theta_{r}\omega)\right]\d r} e^{-2\sigma y(\theta_{\xi}\omega)}\bigg[\|\f(\xi)\|^2_{\L^2(\mathcal{O})}+\|\nabla\land\boldsymbol{f}(\xi)\|_{\mathrm{L}^2(\mathcal{O})}^2\bigg]\d\xi\bigg],
\end{align}
for all $t\geq\tau$. Due to $C^*(\nu,\sigma)>0$, the inequality \eqref{AE10} implies that $\{\v_{\nu}(\cdot)\}_{0<\nu\leq\nu_0}$ is uniformly bounded (that is, the  bound is independent of $\nu$) in $\mathrm{L}^{\infty}(\tau,\tau+T;\V)$. A similar calculation as in \eqref{Derivative} gives $\{\frac{\d \v_{\nu}}{\d t}\}_{0<\nu\leq \nu_0}$ is uniformly bounded in $\mathrm{L}^{2}(\tau,\tau+T;\V')$. Once again, by the Aubin-Lions compactness lemma, there exists a subsequence (still denoted by the same) and $\v\in\mathrm{L}^{2}(\tau,\tau+T;\V)$ with $\frac{\d \v}{\d t}\in\mathrm{L}^{2}(\tau,\tau+T;\V')$ such that (as $\nu\to0$)
\begin{align}
	\v_{\nu}\to&\ \v\ \text{ in }\ 	\mathrm{L}^{2}(\tau,\tau+T;\H) \ \ \text{ strongly},\label{AE11}\\
	\v_{\nu}\xrightharpoonup{w}&\ \v\ \text{ in } \ \mathrm{L}^{2}(\tau,\tau+T;\V) \ \ \text{ weakly}.\label{AE12}
\end{align}
Again, with the help of convergences \eqref{AE11}-\eqref{AE12}, one can pass the limit $\nu\to0$ in \eqref{SNSE} and obtain that $\v(\cdot)$ satisfies \eqref{EuEq} in the distributional sense and \eqref{AE10} holds for $\v(\cdot)$ also. Since $z(\cdot,\omega)$ is continuous, the transformation \eqref{COV} confirms that $\u(\cdot)$ is a global weak (analytic) solution of the system \eqref{1} in the sense of Definition \ref{GWS}. In addition, the transformation \eqref{COV} and the estimate  \eqref{AE10} imply that $\u(\cdot)$ satisfies \eqref{AE}, which completes the proof.
\end{proof}

\subsection{Uniqueness} 
In Subsection \ref{Existence}, we proved  that for the  initial data $\u_{\tau}\in\V$, and external forcing $\f\in\mathrm{L}^2_{\mathrm{loc}}(\R;\L^2(\mathcal{O}))$ and $\nabla\land\f\in\mathrm{L}^2_{\mathrm{loc}}(\R;\mathrm{L}^2(\mathcal{O}))$, there exists a solution for the system \eqref{1} in the sense of Definition \ref{GWS}. In this subsection, we show that if the initial data and external forcing have more regularity, that is, $\u_{\tau}\in\V$, $\nabla\land\u_{\tau}\in\mathrm{L}^{\infty}(\mathcal{O})$, $\f\in\mathrm{L}^2_{\mathrm{loc}}(\R;\L^2(\mathcal{O}))$ and $\nabla\land\f\in\mathrm{L}^{\infty}_{\mathrm{loc}}(\R;\mathrm{L}^{\infty}(\mathcal{O}))$, the solution of the system \eqref{1} obtained in Theorem \ref{Existence-Eu} is unique. In order to prove the uniqueness, the following two lemmas play a crucial role. %Even though the following lemma is for the whole space $\R^2$, but it holds for bounded domains also.
%The idea of proof of the following lemma has been taken from \cite[Corollary 9.11]{Brezis}.
\begin{lemma}[{\cite[Theorem 8.5, part (ii)]{EHL-ML}}]
	For every $\v\in\H^1(\mathcal{O})$, there exists a constant $C>0$ which is independent of $p$, such that 
	\begin{align}\label{SE-2D}
		\|\v\|_{\L^p(\mathcal{O})}\leq C p^{\frac{1}{2}}\|\v\|_{\H^1(\mathcal{O})}, \ \ \ \text{ for all }\ p\in[2,\infty).
	\end{align}
\end{lemma}
Now, since $\mathrm{div\ }\v_{\nu}=0$, we can find a stream function $\boldsymbol{\psi}_{\nu}$ such that
$\v_{\nu}=\nabla^{\perp}\boldsymbol{\psi}_{\nu},$
and $\boldsymbol{\psi}_{\nu}$ satisfies  
\begin{equation*}
	\left\{
	\begin{aligned}
		-\Delta\boldsymbol{\psi}_{\nu}&=\varrho_{\nu}, \ \ \text{ on }\ \mathcal{O}, \\
		\boldsymbol{\psi}_{\nu}|_{\partial\mathcal{O}}&=0.
	\end{aligned}
	\right.
\end{equation*}
Let $\mathcal{G}_{\mathcal{O}}$ be the Green function with homogeneous boundary conditions.
Then, $\boldsymbol{\psi}_{\nu}$ satisfies
\begin{align}\label{GC1}
	\boldsymbol{\psi}_{\nu}(x)=\int_{\mathcal{O}}\mathcal{G}_{\mathcal{O}}(x,\rho)\varrho_{\nu}(\rho)\d \rho \ \ \text{ which implies }\ \  \v_{\nu}(x)=\int_{\mathcal{O}}\nabla^{\perp}\mathcal{G}_{\mathcal{O}}(x,\rho)\varrho_{\nu}(\rho)\d \rho.
\end{align}
It is known that the Green function $\mathcal{G}_{\mathcal{O}}$ and its derivatives satisfies
\begin{equation}\label{GC2}
	\left\{
	\begin{aligned}
		|\mathcal{G}_{\mathcal{O}}(x,z)|&\leq C\big[\log(|x-z|)+1\big],\\
		|\nabla^{\perp}\mathcal{G}_{\mathcal{O}}(x,z)|&\leq C|x-z|^{-1},\\	
		\bigg|\sum_{i=1}^{2}\frac{\partial}{\partial x_i}\nabla^{\perp}\mathcal{G}_{\mathcal{O}}(x,z)\bigg|&\leq C|x-z|^{-2}.
	\end{aligned}
	\right.
\end{equation}
By \eqref{GC1}-\eqref{GC2} and the Calder\'on-Zygmund theorem on singular integral (cf. \cite{Stein}), we infer that
\begin{align}\label{GC3}
	\|\nabla\v_{\nu}\|_{\L^p(\mathcal{O})}\leq C_p\|\varrho_{\nu}\|_{\L^p(\mathcal{O})} \ \ \ \text{ for }\ p\in(1,\infty).
\end{align}
Furthermore, for $p>1$ the constant $C_p$ appearing in \eqref{GC3} can be rewritten as
\begin{align*}
	C_p\leq C\frac{p^2}{p-1},
\end{align*}
where $C$ is a constant independent of $p$. This fact follows from the constant
which is calculated carefully in Marcinkiewicz interpolation inequality (cf. \cite{GT}).  Based on the above discussion, we state the following lemma which can also be found in the works  \cite[Lemma 4.1]{Yudovich1995} and \cite[Lemma 1.2]{Yudovich2005}.
\begin{lemma}\label{Grad-Curl}
	For a function $\v$ satisfing the system \eqref{Elleptic-E}, we have 
	\begin{align}\label{Grad-Curl1}
		\|\nabla\v\|_{\L^p(\mathcal{O})}\leq \frac{C p^2}{p-1}\|\nabla\land\v\|_{\mathrm{L}^{p}(\mathcal{O})},\ \ \text{for all}\ \ \nabla\land\v\in\mathrm{L}^{p}(\mathcal{O}),
	\end{align}
for any $p>1$, where the constant $C>0$ depends on $\mathcal{O}$ but not on $p$.
\end{lemma}

\begin{theorem}\label{Uniqueness}
		Suppose that $\u_{\tau}\in\V$, $\nabla\land\u_{\tau}\in\mathrm{L}^{\infty}(\mathcal{O})$, $\f\in\mathrm{L}^2_{\mathrm{loc}}(\R;\L^2(\mathcal{O}))$ and $\nabla\land\f\in\mathrm{L}^{\infty}_{\mathrm{loc}}(\R;\mathrm{L}^{\infty}(\mathcal{O}))$. Then, the solution of the system \eqref{1} obtained in Theorem \ref{Existence-Eu} is unique. In addition, the solution $\u\in\mathrm{L}^{\infty}(\tau,\tau+T;{\mathbb{W}}^{1,p}(\mathcal{O}))\cap\mathrm{L}^{\infty}(\tau,\tau+T;\L^{\infty}(\mathcal{O}))$ for $p\in[2,\infty)$ and $\nabla\land\u\in\mathrm{L}^{\infty}(\tau,\tau+T;\L^{\infty}(\mathcal{O}))$. Moreover, $\nabla\land\u(\cdot)$ satisfies
		\begin{align}\label{UNS}
			\|\nabla\land\u(t)\|_{\L^{\infty}(\mathcal{O})}&\leq e^{\sigma y(\theta_{t}\omega)}\bigg[e^{\int_{\tau}^{t}\left[-\frac{\sigma^2}{2}+2\sigma y(\theta_{r}\omega)\right]\d r-\sigma y(\theta_{\tau}\omega)}\|\nabla\land\u_{\tau}\|_{\mathrm{L}^{\infty}(\mathcal{O})}\nonumber\\&\quad+\int_{\tau}^{t}e^{\int_{\xi}^{t}\left[-\frac{\sigma^2}{2}+2\sigma y(\theta_{r}\omega)\right]\d r-\sigma y(\theta_{\xi}\omega)}\|\nabla\land\f(\xi)\|_{\mathrm{L}^{\infty}(\mathcal{O})}\d\xi\bigg],
		\end{align}
	for a.e. $t\geq\tau$.
\end{theorem}
\begin{proof}
	The proof is divided into the following two steps:
	\vskip 2mm
	\noindent
	\textbf{Step I:} \textit{Claim:  $\v_{\nu},\ \v\in \mathrm{L}^{\infty}(\tau,\tau+T;{\mathbb{W}}^{1,p}(\mathcal{O}))\cap\mathrm{L}^{\infty}(\tau,\tau+T;\L^{\infty}(\mathcal{O}))$ for $p\in[2,\infty)$.}
	
	It implies from \eqref{SE-2D} and \eqref{AE10} that 
	\begin{align}\label{UN1}
		\|\v_{\nu}\|_{\mathrm{L}^{\infty}(\tau,\tau+T;{\mathbb{L}}^{p}(\mathcal{O}))}&\leq p^{\frac{1}{2}} C(\omega) \|\v_{\nu}\|_{\mathrm{L}^{\infty}(\tau,\tau+T;\V)}\nonumber\\& \leq p^{\frac{1}{2}} C(T, \omega,\|\v_{\tau}\|_{\V},\|\f\|_{\mathrm{L}^{2}(\tau,\tau+T;\L^2(\mathcal{O}))},\|\nabla\land\f\|_{\mathrm{L}^{2}(\tau,\tau+T;\mathrm{L}^2(\mathcal{O}))}), 
	\end{align}
for $p\in[2,\infty)$. Taking the inner product of the first equation of \eqref{C-SNSE} with $|\varrho_{\nu}|^{p-2}\varrho_{\nu}$, for  $p\in[2,\infty)$, we obtain 
	\begin{align}\label{UN12}
	\frac{1}{p}&\frac{\d}{\d t}\|\varrho_{\nu}(t)\|^p_{\mathrm{L}^p(\mathcal{O})}+p\nu\||\nabla\varrho_{\nu}(t)||\varrho_{\nu}(t)|^{\frac{p-2}{2}}\|^2_{\mathrm{L}^2(\mathcal{O})}+\frac{\sigma^2}{2}\|\varrho_{\nu}(t)\|^p_{\mathrm{L}^p(\mathcal{O})}\nonumber\\&=e^{-\sigma y(\theta_t\omega)}\int_{\mathcal{O}}(\nabla\land\boldsymbol{f}(t,x))|\varrho_{\nu}(t,x)|^{p-2}\varrho_{\nu}(t,x)\d x+\sigma y(\theta_{t}\omega)\|\varrho_{\nu}(t)\|^p_{\mathrm{L}^p(\mathcal{O})}\nonumber\\&\leq  e^{-\sigma y(\theta_t\omega)}\|\nabla\land\boldsymbol{f}(t)\|_{\mathrm{L}^p(\mathcal{O})}\|\varrho_{\nu}(t)\|_{\mathrm{L}^p(\mathcal{O})}^{p-1}+\sigma |y(\theta_{t}\omega)|\|\varrho_{\nu}(t)\|^p_{\mathrm{L}^p(\mathcal{O})}\nonumber\\&\leq C e^{-p\sigma y(\theta_t\omega)}\|\nabla\land\f(t)\|^p_{\mathrm{L}^p(\mathcal{O})}+(1+\sigma|y(\theta_{t}\omega)|)\|\varrho_{\nu}(t)\|^p_{\mathrm{L}^p(\mathcal{O})},
	\end{align}
for a.e. $t\geq \tau$, which implies (due to Gronwall's inequality)
\begin{align*}
	&\sup_{t\in[\tau,\tau+T]}\|\varrho_{\nu}(t)\|_{\mathrm{L}^p(\mathcal{O})}^p\nonumber\\&\leq p\bigg[\frac{1}{p} \|\nabla\land\v_{\tau}\|_{\mathrm{L}^p(\mathcal{O})}^p +C\int_{\tau}^{\tau+T}e^{-p\sigma y(\theta_{\xi}\omega)}\|\nabla\land\f(\xi)\|^p_{\mathrm{L}^p(\mathcal{O})}\d\xi\bigg]e^{p\int_{\tau}^{\tau+T}(1+\sigma |y(\theta_{r}\omega)|)\d r}.
\end{align*}
Therefore, we infer
\begin{align*}
	\sup_{t\in[\tau,\tau+T]}\|\varrho_{\nu}(t)\|_{\mathrm{L}^p(\mathcal{O})}\leq p^{\frac{1}{p}}C(T,\omega,\mathcal{O},\|\nabla\land\u_{\tau}\|_{\mathrm{L}^{\infty}(\mathcal{O})},\|\nabla\land\f\|_{\mathrm{L}^{\infty}(\tau,\tau+T;\mathrm{L}^{\infty}(\mathcal{O}))}).
\end{align*}
In view of Lemma \ref{Grad-Curl}, we deduce
\begin{align}\label{UN2}
		\|\nabla\v_{\nu}\|_{\mathrm{L}^{\infty}(\tau,\tau+T;{\mathbb{L}}^{p}(\mathcal{O}))}\leq \frac{p^{\frac{1}{p}}p^2}{p-1}C(T,\omega,\mathcal{O},\|\nabla\land\u_{\tau}\|_{\mathrm{L}^{\infty}(\mathcal{O})},\|\nabla\land\f\|_{\mathrm{L}^{\infty}(\tau,\tau+T;\mathrm{L}^{\infty}(\mathcal{O}))}),
	\end{align}
for $p\in[2,\infty)$. Since the bounds in \eqref{UN1} and \eqref{UN2} are independent of $\nu$, one can obtain by passing limit $\nu\to0$ that 
\begin{equation}\label{UN3}
	\left\{
	\begin{aligned}
		\|\v\|_{\mathrm{L}^{\infty}(\tau,\tau+T;{\mathbb{L}}^{p}(\mathcal{O}))}& \leq p^{\frac{1}{2}} C(T, \omega,\|\u_{\tau}\|_{\V},\|\f\|_{\mathrm{L}^{2}(\tau,\tau+T;\V)}), \\
	\|\nabla\v\|_{\mathrm{L}^{\infty}(\tau,\tau+T;{\mathbb{L}}^{p}(\mathcal{O}))}&\leq \frac{p^{\frac{1}{p}}p^2}{p-1}C(T,\omega,\mathcal{O},\|\nabla\land\u_{\tau}\|_{\mathrm{L}^{\infty}(\mathcal{O})},\|\nabla\land\f\|_{\mathrm{L}^{\infty}(\tau,\tau+T;\mathrm{L}^{\infty}(\mathcal{O}))}),
	\end{aligned}
	\right.
\end{equation}
for $p\in[2,\infty)$. Moreover, Sobolev's inequality  provides $\v_{\nu}, \v\in\mathrm{L}^{\infty}(\tau,\tau+T;\L^{\infty}(\mathcal{O}))$. Let us now apply the maximum principle (see \cite[Subsection 7.1.4]{Evans}) for $\varrho_{\nu}$, which is a solution of \eqref{C-SNSE}, to deduce 
\begin{align}\label{UN45}
	&\|\varrho_{\nu}(t)\|_{\mathrm{L}^{\infty}(\mathcal{O})}\nonumber\\&\leq e^{\int_{\tau}^{t}\left[-\frac{\sigma^2}{2}+2\sigma y(\theta_{r}\omega)\right]\d r}\|\varrho_{\nu}(\tau)\|_{\mathrm{L}^{\infty}(\mathcal{O})}+\int_{\tau}^{t}e^{\int_{\xi}^{t}\left[-\frac{\sigma^2}{2}+2\sigma y(\theta_{r}\omega)\right]\d r}e^{-\sigma y(\theta_{\xi}\omega)}\|\nabla\land\f(\xi)\|_{\mathrm{L}^{\infty}(\mathcal{O})}\d\xi,
\end{align}
for all $t\geq \tau$, which immediately gives
\begin{align}
	%\|\nabla\land\v_{\nu}\|_{\mathrm{L}^{\infty}(\tau,\tau+T;\L^{\infty}(\mathcal{O}))}=
	\|\varrho_{\nu}\|_{\mathrm{L}^{\infty}(\tau,\tau+T;\mathrm{L}^{\infty}(\mathcal{O}))}\leq C(T,\omega)\big[\|\nabla\land\v_{\tau}\|_{\mathrm{L}^{\infty}(\mathcal{O})}+\|\nabla\land\f\|_{\mathrm{L}^{\infty}(\tau,\tau+T;\mathrm{L}^{\infty}(\mathcal{O}))}\big].
\end{align}
Therefore, from \eqref{COV} and \eqref{UN45}, one can easily deduce \eqref{UNS}. 
	\vskip 2mm
\noindent
\textbf{Step II:} \textit{The solution of the system \eqref{1} obtained in Theorem \ref{Existence-Eu} is unique.} Let $\v_1(\cdot)$ and $\v_2(\cdot)$ be two solutions of the system \eqref{EuEq} with $\v_1(\tau)=\v_2(\tau)=\v_{\tau}$. Then $\v(\cdot):=\v_1(\cdot)-\v_2(\cdot)$ satisfies  for a.e. $t\geq \tau$ and $ \boldsymbol{\phi}\in\V$
\begin{align}
		\left\langle\frac{\d\v}{\d t},\boldsymbol{\phi}\right\rangle+\left[\frac{\sigma^2}{2}-y(\theta_{t}\omega)\right](\v,\boldsymbol{\phi})+e^{\sigma y(\theta_{t}\omega)}\langle(\v\cdot\nabla)\v_1,\boldsymbol{\phi}\rangle+e^{\sigma y(\theta_{t}\omega)}\langle(\v_2\cdot\nabla)\v,\boldsymbol{\phi}\rangle=0. 
\end{align}
Taking $\boldsymbol{\phi}=\v$, we attain (for $p>3$)
\begin{align}\label{UN4}
	\frac{\d}{\d t}\|\v(t)\|^2_{\H}&=-\left[\sigma^2-2\sigma y(\theta_{t}\omega)\right]\|\v(t)\|^2_{\H}-2e^{\sigma y(\theta_{t}\omega)}\langle(\v(t)\cdot\nabla)\v_1(t),\v(t)\rangle\nonumber\\&\leq 2e^{\sigma y(\theta_{t}\omega)} \|\nabla\v_1(t)\|_{\L^p(\mathcal{O})}\|\v(t)\|^2_{\L^{\frac{2p}{p-1}}(\mathcal{O})}+2\sigma|y(\theta_{t}\omega)|\|\v(t)\|^2_{\H}\nonumber\\&\leq 2e^{\sigma y(\theta_{t}\omega)} \|\nabla\v_1(t)\|_{\L^p(\mathcal{O})}\|\v(t)\|^{\frac{2}{p-2}}_{\L^{p}(\mathcal{O})}\|\v(t)\|^{\frac{2(p-3)}{p-2}}_{\H}+2\sigma|y(\theta_{t}\omega)|\|\v(t)\|^2_{\H},
\end{align}
for a.e. $t\geq\tau$. We also have 
\begin{align}\label{UN5}
	\frac{\d}{\d t}\|\v(t)\|^2_{\H}&=\frac{\d}{\d t}\left[\|\v(t)\|^{\frac{2}{p-2}}_{\H}\right]^{p-2}=(p-2)\|\v(t)\|_{\H}^{\frac{2(p-3)}{p-2}}\frac{\d}{\d t}\|\v(t)\|^{\frac{2}{p-2}}_{\H},
\end{align}
for a.e. $t\geq\tau$, which implies from \eqref{UN4} and 
\iffalse 
 that 
\begin{align}
	\|\v(t)\|_{\H}^{\frac{2(p-3)}{p-2}}\bigg[(p-2)\frac{\d}{\d t}\|\v(t)\|^{\frac{2}{p-2}}_{\H}-2e^{\sigma y(\theta_{t}\omega)} \|\nabla\v_1(t)\|_{\L^p(\mathcal{O})}\|\v(t)\|^{\frac{2}{p-2}}_{\L^{p}(\mathcal{O})}-2\sigma|y(\theta_{t}\omega)|\|\v(t)\|^{\frac{2}{p-2}}_{\H}\bigg]\leq 0,
\end{align}
for a.e. $t\geq\tau$. Since
\fi 
 $\|\v\|_{\H}^{\frac{2(p-3)}{p-2}}\geq0$ that 
\begin{align}\label{UN6}
	\frac{\d}{\d t}\|\v(t)\|^{\frac{2}{p-2}}_{\H}\leq\frac{2e^{\sigma y(\theta_{t}\omega)}}{p-2} \|\nabla\v_1(t)\|_{\L^p(\mathcal{O})}\|\v(t)\|^{\frac{2}{p-2}}_{\L^{p}(\mathcal{O})}+\frac{2\sigma|y(\theta_{t}\omega)|}{p-2}\|\v(t)\|^{\frac{2}{p-2}}_{\H},
\end{align}
for a.e. $t\geq\tau$. Now, applying Gronwall's inequality in \eqref{UN6} from $\tau$ to $t$ with $t\geq\tau$, we infer
\begin{align}\label{UN35}
	\|\v(t)\|^2_{\H}&\leq    \left[\frac{2}{p-2}\int_{\tau}^{t}e^{\sigma y(\theta_{\xi}\omega)} \|\nabla\v_1(\xi)\|_{\L^p(\mathcal{O})}\|\v(\xi)\|^{\frac{2}{p-2}}_{\L^{p}(\mathcal{O})}\d\xi\right]^{p-2}e^{2\sigma\int_{\tau}^{t}|y(\theta_{\xi}\omega)|\d\xi}
	\nonumber\\%&\leq [C(t-\tau)]^{p-2} \left[\frac{p^{\frac{1}{p}}p^2}{(p-1)(p-2)}\right]^{p-2} (Cp)^2\nonumber\\
	&\leq [C(t-\tau)]^{p-2}(Cp)^2 \frac{1}{p^{\frac{2}{p}}}\left[\frac{p^2}{(p-1)(p-2)}\right]^{p-2}e^{2\sigma\int_{\tau}^{t}|y(\theta_{\xi}\omega)|\d\xi},
\end{align}
where we have used the bounds obtained in \eqref{UN3} and the  constant $C$ appearing in \eqref{UN35} $$C=C(T,\omega,\mathcal{O},\|\u_{\tau}\|_{\V},\|\f\|_{\mathrm{L}^{2}(\tau,\tau+T;\L^2(\mathcal{O}))},\|\nabla\land\u_{\tau}\|_{\mathrm{L}^{\infty}(\mathcal{O})},\|\nabla\land\f\|_{\mathrm{L}^{\infty}(\tau,\tau+T;\mathrm{L}^{\infty}(\mathcal{O}))})>0,$$ is independent of $p$. Let us now take $t$ sufficiently small such that $C(t-\tau)<1$. Further, it provides us that $\lim\limits_{p\to+\infty}[C(t-\tau)]^{p-2}(Cp)^2e^{2\sigma\int_{\tau}^{t}|y(\theta_{\xi}\omega)|\d\xi}=0,$ while $\lim\limits_{p\to+\infty}p^{-\frac{2}{p}}\left[\frac{p^2}{(p-1)(p-2)}\right]^{p-2}=e^3$. Hence $\v(t)=0$ if $C(t-\tau)<1$. We repeat the same arguments a finite number of times to cover the whole interval $[\tau,\tau+T]$, so that the uniqueness is proved.
 \end{proof}

\begin{remark}
	Using Moser estimates, Kato-Ponce commutator estimates (see \cite{MB}), and Beale-Kato-Majda type logarithmic Sobolev  inequality (see \cite[Lemma 2.1]{HV} and \cite[Section 9]{HV}), the existence of a \textsl{global smooth solution} for the system \eqref{1} (for $\f=\mathbf{0}$) has been established in \cite[Theorem 4.6, part(1)]{HV} for the initial data $$\u_\tau\in \mathbb{X}_{m,p}:=\left\{\u\in\mathbb{W}^{m,p}(\mathcal{O}):\nabla\cdot\u=0 \ \text{ and }\ \u\cdot\boldsymbol{n}\big|_{\partial\mathcal{O}}=0\right\}, $$ where $ m>2/p+1$, $p\in[2,\infty)$. The existence of global smooth solution is obtained by showing  that $\sup\limits_{t\in[0,T]}\|\u(t)\|_{\mathbb{W}^{1,\infty}}<\infty$, $\mathbb{P}$-a.s. Therefore, if we take our initial data $\u_\tau\in \mathbb{X}_{m,p},$ for $ m>2/p+1$ with  $\nabla\land\u_{\tau}\in\mathrm{L}^{\infty}(\mathcal{O})$, then the unique weak solution of the system \eqref{1} obtained in Theorems \ref{Existence-Eu} and \ref{Uniqueness} is smooth satisfying $	\u(\cdot)\in\mathrm{C}([\tau,\infty);\mathbb{X}_{m,p})$, $\mathbb{P}$-a.s. 	
	
	If $\boldsymbol{u}\in \mathrm{L}^{\infty}(0,T;\mathbb{W}^{1,\infty}(\mathcal{O}))$, $\mathbb{P}$-a.s., 	 then one can easily obtain the continuous dependence on the data also. In fact an exponential stability result for $\sigma>0$. For instance, an estimate similar to \eqref{UN4} yields 
	\begin{align}
		\frac{\mathrm{d}}{\mathrm{d}t}\|\boldsymbol{v}(t)\|_{\mathbb{H}}^2+\sigma^2\|\boldsymbol{v}(t)\|_{\mathbb{H}}^2&=-2e^{\sigma y(\theta_{t}\omega)}\langle(\boldsymbol{v}(t)\cdot\nabla)\boldsymbol{v}_1(t),\boldsymbol{v}(t)\rangle+2\sigma y(\theta_{t}\omega)\|\boldsymbol{v}(t)\|_{\mathbb{H}}^2\nonumber\\&\leq2\bigg[e^{\sigma y(\theta_{t}\omega)}\|\nabla\boldsymbol{v}_1(t)\|_{\mathbb{L}^{\infty}}+\sigma |y(\theta_{t}\omega)|\bigg]\|\boldsymbol{v}(t)\|_{\mathbb{H}}^2,
	\end{align}
	for a.e. $t\in[\tau,\tau+T]$. Therefore variation of constants formula gives 
	\begin{align}
		\|\boldsymbol{v}(t)\|_{\mathbb{H}}^2\leq e^{-\sigma^2(t-\tau)}\|\boldsymbol{v}_{\tau}\|_{\mathbb{H}}^2\exp\left(2\left[\int_{\tau}^te^{\sigma y(\theta_{\xi}\omega)}\|\nabla\boldsymbol{v}_1(\xi)\|_{\mathbb{L}^{\infty}}+\sigma |y(\theta_{\xi}\omega)|\right]\mathrm{d}\xi\right), \ \mathbb{P}\text{-a.s.,}%\leq Ce^{-\sigma^2(t-\tau)}\|\boldsymbol{v}_{\tau}\|_{\mathbb{H}}^2, 
	\end{align}
	for all $t\in[\tau,\tau+T]$, where $\boldsymbol{v}_{\tau}=\boldsymbol{v}_1(\tau)-\boldsymbol{v}_2(\tau)$. 
\end{remark}

\section{Pullback stochastic weak attractor and stochastic weak attractor}\label{sec4}\setcounter{equation}{0}

In this section, we demonstrate the existence of \textsl{pullback stochastic weak attractors (stochastic weak attractors)} for 2D non-autonomous (autonomous) stochastic EE driven by a linear multiplicative white noise. %(that is, $\f\in\V$ and $\nabla\land\f\in\mathrm{L}^{\infty}(\mathcal{O})$). 
 We apply the abstract theory (Theorems \ref{SWA} and \ref{SWA1}) developed in \cite{HB2000} to prove the main result of this section.

\subsection{Existence of pullback stochastic weak attractor}
\iffalse 
In order to prove the existence of pullback stochastic weak attractor, we need the following hypothesis on the time-dependent external forcing $\f$:
	\begin{hypothesis}\label{Hyp-f}
	For the time-dependent external forcing $\f\in\mathrm{L}^{2}_{\emph{loc}}(\R;\L^2(\mathcal{O}))$ and $\nabla\land\f\in\mathrm{L}^{\infty}_{\mathrm{loc}}(\R;\mathrm{L}^{\infty}(\mathcal{O}))$, there exist two numbers $\delta_1,  \delta_2\in[0,\frac{\sigma^2}{2})$ such that for all $t\in\R$
	\begin{align}\label{forcing1}
	\int_{-\infty}^{t} e^{\delta_1\xi}\bigg[\|\f(\xi)\|^2_{\L^2(\mathcal{O})}+\|\nabla\land\f(\xi)\|^2_{\mathrm{L}^{2}(\mathcal{O})}\bigg]\d \xi+	\int_{-\infty}^{t} e^{\delta_2\xi}\|\nabla\land\f(\xi)\|_{\mathrm{L}^{\infty}(\mathcal{O})}\d \xi<\infty.
\end{align}
\end{hypothesis}
\fi 
Let $\delta$, $\d$ and $\d_{\infty}$ be  the metrics induced by the norms on $\H$, $\V$ and $\mathrm{L}^{\infty}(\mathcal{O})$, respectively. We denote by 
\begin{align}\label{Wspace}
	\mathcal{W}=\{\v\in\V: \nabla\land\v\in\mathrm{L}^{\infty}(\mathcal{O})\}.
\end{align}
Also, let $\d_{\mathcal{W}}$ be the metric on $\mathcal{W}$ which is defined by
\begin{align*}
	\d_{\mathcal{W}}(\u,\v)=\d(\u,\v)+\d_{\infty}(\nabla\land\u,\nabla\land\v)=\|\u-\v\|_{\V}+\|\nabla\land\u-\nabla\land\v\|_{\mathrm{L}^{\infty}(\mathcal{O})}.
\end{align*}
Theorems \ref{Existence-Eu} and \ref{Uniqueness} allow us to define a family $\{S(t,\tau,\omega)\}_{t\geq\tau, \omega\in\Omega}$ by $S(t,\tau,\omega):\mathcal{W}\to\mathcal{W}$ such that 
\begin{align}
	S(t,\tau,\omega)\u_\tau:=\u(t;\tau,\omega,\u_{\tau}).
\end{align}
%Our next aim is to prove that the family $\{S(t,\tau,\omega)\}_{t\geq \tau, \omega\in\Omega}$ satisfies Hypotheses \ref{H1} and \ref{H2}. 
The following lemma proves that the family $\{S(t,\tau,\omega)\}_{t\geq \tau, \omega\in\Omega}$ satisfies Hypothesis \ref{H1}.

\begin{lemma}\label{H111}
	The family of mappings $\{S(t,\tau,\omega)\}_{t\geq \tau, \omega\in\Omega}$ associated to the system \eqref{1} satisfies Hypothesis \ref{H1}.
\end{lemma}
\begin{proof}
	The evolution property holds due to the uniqueness of the solution in $\mathcal{W}$ (see Theorem \ref{Uniqueness}). Now, fix $\omega\in \Omega$ and take $t\geq \tau$. Let $\{\w^n_{\tau}\}_{n\in\N}$ be a sequence which is $\d_{\mathcal{W}}$-bounded and $\delta$-convergent to $\w_{\tau}$ in $\mathcal{W}$, that is, there exists a positive number $K$ such that 
	\begin{align}\label{H10}
		\|\w_{\tau}^n\|_{\mathcal{W}}\leq K, \ \text{ for all } \ n\in\N\ \text{ and }\ \lim_{n\to+\infty}\|\w_{\tau}^n-\w_{\tau}\|_{\H}=0.
	\end{align}
	 Let $\w^n(\cdot)$ and $\w(\cdot)$ be the unique solutions to the system \eqref{1} corresponding to the initial data $\w^n_{\tau}$ and $\w_{\tau}$, respectively. Then $\mathscr{W}_n(\cdot):=\w^n(\cdot)-\w(\cdot)$ satisfies for a.e. $t\in[\tau,\tau+T]$ and $ \boldsymbol{\phi}\in\V$
	\begin{align*}
		&\left\langle\frac{\d\mathscr{W}_n}{\d t},\boldsymbol{\phi}\right\rangle+\left[\frac{\sigma^2}{2}-\sigma y(\theta_{t}\omega)\right](\mathscr{W}_n,\boldsymbol{\phi})=-e^{\sigma y(\theta_{t}\omega)}\langle(\mathscr{W}_n\cdot\nabla)\w,\boldsymbol{\phi}\rangle-e^{\sigma y(\theta_{t}\omega)}\langle(\w^n\cdot\nabla)\mathscr{W}_n,\boldsymbol{\phi}\rangle.
	\end{align*}
	Taking $\boldsymbol{\phi}=\mathscr{W}_n$ and arguing similarly as in the proof of Theorem \ref{Uniqueness} (see \eqref{UN6}), we obtain (for $p>3$) 
	\begin{align}\label{H11}
		&\frac{\d}{\d t}\|\mathscr{W}_n(t)\|^{\frac{2}{p-2}}_{\H}\nonumber\\&\leq\frac{2e^{\sigma y(\theta_{t}\omega)}}{p-2} \|\nabla\w(t)\|_{\L^p(\mathcal{O})}\|\mathscr{W}_n(t)\|^{\frac{2}{p-2}}_{\L^{p}(\mathcal{O})}+\frac{2\sigma|y(\theta_{t}\omega)|}{p-2}\|\mathscr{W}_n(t)\|^{\frac{2}{p-2}}_{\H}\nonumber\\&\leq \frac{p^2Ce^{\sigma y(\theta_{t}\omega)}}{(p-1)(p-2)}\|\nabla\land\w(t)\|_{\mathrm{L}^p(\mathcal{O})} \left[p^{\frac{1}{2}}C\|\mathscr{W}_n(t)\|_{\V}\right]^{\frac{2}{p-2}}+\frac{2\sigma|y(\theta_{t}\omega)|}{p-2}\|\mathscr{W}_n(t)\|^{\frac{2}{p-2}}_{\H}\nonumber\\& \leq  \frac{p^2C(pC)^{\frac{1}{p-2}}e^{\sigma y(\theta_{t}\omega)}}{(p-1)(p-2)} \|\nabla\land\w(t)\|_{\mathrm{L}^{\infty}(\mathcal{O})} \bigg[\|\w^n(t)\|_{\V}+\|\w(t)\|_{\V}\bigg]^{\frac{2}{p-2}}+\frac{2\sigma|y(\theta_{t}\omega)|}{p-2}\|\mathscr{W}_n(t)\|^{\frac{2}{p-2}}_{\H}\nonumber\\& \leq  \frac{p^2C_{T}(pC_{K,T})^{\frac{1}{p-2}}e^{\sigma y(\theta_{t}\omega)}}{(p-1)(p-2)}+\frac{2\sigma|y(\theta_{t}\omega)|}{p-2}\|\mathscr{W}_n(t)\|^{\frac{2}{p-2}}_{\H}, 
	\end{align}
	for a.e. $t\geq\tau$, where we have used inequalities \eqref{AE10}, \eqref{SE-2D} and \eqref{Grad-Curl1}, and the fact that $\w^n,  \w \in \mathcal{W}$. Also, the constants $C_{T}$ and $C_{K,T}$ are independent of $p$. Now, integration of \eqref{H11} from $\tau$ to $t$ with $t\geq\tau$ yields
	\begin{align}\label{H22}
		\|\mathscr{W}_n(t)\|^{2}_{\H}&\leq \bigg[\|\w^n_{\tau}-\w_{\tau}\|^{\frac{2}{p-2}}_{\H}  +  \frac{p^2C_{T}(pC_{K,T})^{\frac{1}{p-2}}}{(p-1)(p-2)} \int_{\tau}^{t}e^{\sigma y(\theta_{\xi}\omega)}\d\xi\bigg]^{p-2}e^{2\sigma\int_{\tau}^{t}|y(\theta_{\xi}\omega)|\d\xi}.
	\end{align}
\iffalse 
Since $\|\w^n_{\tau}-\w_{\tau}\|_{\H}\to 0$ as $n\to+\infty$, there exists an $N_0>0$ such that 
\begin{align*}
	\|\w^n_{\tau}-\w_{\tau}\|^{\frac{2}{p-2}}_{\H}  \leq  \frac{p^2C_{T}(pC_{K,T})^{\frac{2}{p-2}}}{(p-1)(p-2)} \int_{\tau}^{t}\frac{2}{z(\xi,\omega)}\d\xi, \ \ \ \ \ \ \ \text{for all } \ n\geq N_0,
\end{align*}
which gives from \eqref{H22} that
\fi 
Further, $\lim\limits_{n\to+\infty}\|\w_{\tau}^n-\w_{\tau}\|_{\H}=0$ provides for all $t\geq\tau $
	\begin{align}\label{H33}
	\lim_{n\to+\infty}\|\mathscr{W}_n(t)\|^{2}_{\H}&\leq pC_{K,T} \bigg[C_{T}\int_{\tau}^{t}e^{\sigma y(\theta_{\xi}\omega)}\d\xi\bigg]^{p-2}   \left[\frac{p^2}{(p-1)(p-2)}\right]^{p-2}e^{2\sigma\int_{\tau}^{t}|y(\theta_{\xi}\omega)|\d\xi},
\end{align}
	where the constants $C_{K,T}$ and $C_{T}$ appearing in \eqref{H33} are independent of $p$. Let us now take $t$ sufficiently small such that $C_{T}\int_{\tau}^{t}e^{\sigma y(\theta_{\xi}\omega)}\d\xi<1$. It also implies that $$\lim\limits_{p\to+\infty}pC_{K,T}\left[C_{T}\int_{\tau}^{t}e^{\sigma y(\theta_{\xi}\omega)}\d\xi\right]^{p-2}e^{2\sigma\int_{\tau}^{t}|y(\theta_{\xi}\omega)|\d\xi}=0 \ \ \ \text{ while }\ \ \lim\limits_{p\to+\infty}\left[\frac{p^2}{(p-1)(p-2)}\right]^{p-2}=e^3.$$ Hence $\lim\limits_{n\to+\infty}\|\w^n(t)-\w(t)\|_{\H}=0,$ for all $t\in(\tau,\tau+t_0)$ such that $C_{T}\int_{\tau}^{\tau+t_0}e^{\sigma y(\theta_{\xi}\omega)}\d\xi<1$. We repeat the same arguments a finite number of times to cover the whole interval $[\tau,\tau+T]$. Hence part (ii)  of Hypothesis \ref{H1} is satisfied.
\end{proof}

Our next aim is to show that the mapping $(\tau,\omega)\mapsto S(t,\tau,\omega)\w$ satisfies Hypothesis \ref{H2}. The following lemma shows that for all $t\in\R$ and $\w\in\mathcal{W}$, the mapping $(\tau,\omega)\mapsto S(t,\tau,\omega)\w$ is measurable from $((-\infty,t],\mathscr{B}((-\infty,t])$ to $(\mathcal{W},\mathscr{B}(\mathbb{H})\cap\mathcal{W})$.

\begin{lemma}\label{H222}
	For a given $\u^*_{\tau}\in\mathcal{W}$, $t\in\R$ and $\omega\in\Omega$, the $\mathcal{W}$-valued mapping $\tau\mapsto S(t,\tau,\omega)\u^*_{\tau}$ is continuous from $(-\infty,t]$ to $\H$. 
\end{lemma}
\begin{proof}
	First we show that, for a given $\u^*_{\tau}\in\mathcal{W}$, $t\in\R$ and $\omega\in\Omega$, the $\mathcal{W}$-valued mapping $\tau\mapsto S(t,\tau,\omega)\u^*_{\tau}$ is right continuous from $(-\infty,t]$ to $\H$. Since $$S(t,\tau,\omega)\u^*_{\tau}=\u(t;\tau,\omega,\u^*_{\tau})=\v(t;\tau,\omega,\v^*_{\tau})e^{y(\theta_{t}\omega)},$$ it is enough to prove that $\v(t;\cdot,\omega,\v^*_{\tau})$ with $\v^*_{\tau}=\u^*_{\tau}e^{y(\theta_{\cdot}\omega)}$ is right continuous from $(-\infty,t]$ to $\H$. Now fix $\tau^*\in\R$, and take $\v^*_{\tau}\in\mathcal{W}$ and $\omega\in\Omega$. It is enough to show that for any given $\varepsilon>0$, there exists a real number $\varepsilon^*=\varepsilon^*(\varepsilon,\tau^*,\omega,\v^*_{\tau})>0$ such that 
	\begin{align}\label{RLTT1}
		\|\v(\xi;\tau,\omega,\v^*_{\tau})-\v^*_{\tau}\|_{\mathbb{H}}<\varepsilon \ \text{ whenever }\  \tau\in(\tau^*,\tau^*+\varepsilon^*), \ \xi\in(\tau,\tau^*+\varepsilon^*),
	\end{align}
	where $\v(\xi;\tau,\omega,\v^*_{\tau})$ is the solution of the system \eqref{1} with the initial data $\v^*_{\tau}$ and initial time $\tau$. 
	\iffalse 
	From \eqref{Grad-Curl1} (for $p=2$), it implies that 
	\begin{align}
		\|\v(\xi;\tau,\omega,\v^*_{\tau})-\v^*_{\tau}\|^2_{\mathcal{W}}&=\|\v(\xi;\tau,\omega,\v^*_{\tau})-\v^*_{\tau}\|^2_{\V}+\|\varrho(\xi;\tau,\omega,\v^*_{\tau})-\nabla\land\v^*_{\tau}\|^2_{\mathrm{L}^{\infty}(\mathcal{O})}\nonumber\\&\leq C\|\varrho(\xi;\tau,\omega,\v^*_{\tau})-\nabla\land\v^*_{\tau}\|^2_{\mathrm{L}^2(\mathcal{O})}+\|\varrho(\xi;\tau,\omega,\v^*_{\tau})-\nabla\land\v^*_{\tau}\|^2_{\mathrm{L}^{\infty}(\mathcal{O})}\nonumber\\&\leq L\|\varrho(\xi;\tau,\omega,\v^*_{\tau})-\nabla\land\v^*_{\tau}\|^2_{\mathrm{L}^{\infty}(\mathcal{O})},
	\end{align}
where $L>0$ is some positive constant. 
\fi 
We have
	\begin{align}\label{RLTT2}
		\|\v(\xi;\tau,\omega,\v^*_{\tau})-\v^*_{\tau}\|^2_{\H}&=\|\v(\xi;\tau,\omega,\v^*_{\tau})\|^2_{\H}-\|\v^*_{\tau}\|^2_{\H}-2(\v(\xi;\tau,\omega,\v^*_{\tau})-\v^*_{\tau},\v^*_{\tau})\nonumber\\&=\int_{\tau}^{\xi}\frac{\d}{\d\zeta}\|\v(\zeta;\tau,\omega,\v^*_{\tau})\|^2_{\H}\d\zeta-2(\v(\xi;\tau,\omega,\v^*_{\tau})-\v^*_{\tau},\v^*_{\tau}).
	\end{align}
	From \eqref{Div-free}, we have 
	\begin{align}\label{ei2}
		\frac{\d}{\d \zeta}\|\v(\zeta)\|^2_{\H}+\frac{\sigma^2}{2}\|\v(\zeta)\|^2_{\H}\leq \frac{2e^{-2\sigma y(\theta_{\zeta}\omega)}}{\sigma^2}\|\f(\zeta)\|^2_{\H}+2\sigma y(\theta_{\zeta}\omega)\|\v(\zeta)\|^2_{\H},
	\end{align}
which implies
\begin{align}\label{RLTT4}
	&\sup_{\tau\in[\tau^*-1,\tau^*+1]}\|\v(\zeta;\tau,\omega,\v^*_{\tau})\|^2_{\H}\nonumber\\&\leq\sup_{\tau\in[\tau^*-1,\tau^*+1]}\bigg[\|\v^*_{\tau}\|^2_{\H}+\frac{2}{\sigma^2}\int_{\tau}^{\zeta}e^{-2\sigma y(\theta_{r}\omega)}\|\f(r)\|^2_{\H}\d r\bigg]e^{2\sigma\int_{\tau}^{\zeta}| y(\theta_{r}\omega)|\d r}\nonumber\\&\leq\left[\|\v^*_{\tau}\|^2_{\H}+\frac{2}{\sigma^2}\int_{\tau^*-1}^{\zeta}e^{-2\sigma y(\theta_{r}\omega)}\|\f(r)\|^2_{\H}\d r\right]e^{2\sigma\int_{\tau^*-1}^{\zeta}| y(\theta_{r}\omega)|\d r},
\end{align}
which is finite and the right hand side is independent of $\tau$. From \eqref{ei2}, we infer
	\begin{align}\label{RLTT3}
		\int_{\tau}^{\xi}\frac{\d}{\d\zeta}\|\v(\zeta;\tau,\omega,\v^*_{\tau})\|^2_{\H}\d\zeta\leq \frac{2}{\sigma^2}\int_{\tau}^{\xi}z^2(\zeta,\omega)\|\f(\zeta)\|^2_{\H}\d\zeta+2\sigma\int_{\tau}^{\xi}|y(\theta_{\zeta}\omega)|\|\v(\zeta;\tau,\omega,\v^*_{\tau})\|^2_{\H}\d\zeta.
	\end{align}
It follows from \eqref{RLTT4}-\eqref{RLTT3} that there exists a real number $\varepsilon^*_1=\varepsilon^*_1(\varepsilon,\tau^*,\omega,\v^*_{\tau})>0$ such that
	\begin{align}\label{RLTT5}
		\int_{\tau}^{\xi}\frac{\d}{\d\zeta}\|\v(\zeta;\tau,\omega,\v^*_{\tau})\|^2_{\H}\d\zeta\leq \frac{\varepsilon^2}{2},  \ \text{ whenever }\  \tau\in(\tau^*,\tau^*+\varepsilon^*_1), \ \xi\in(\tau,\tau^*+\varepsilon^*_1),
	\end{align}
by the absolute continuity of the Lebesgue integral. 
\iffalse
	Now we estimate the final term on the right hand side of \eqref{RLTT2}. From the fact that $$\v(\cdot;\xi,\omega,\v^*_{\tau})\in\mathrm{C}([\xi,+\infty);\H)\cap\mathrm{L}^2_{\mathrm{loc}}(\xi,+\infty;\V),$$ and \eqref{RLTT4}, we can find a positive number $K(\omega,\xi^*,\v^*_{\tau})$ which is independent of $\xi$ such that 
	\begin{align}\label{RLTT6}
		\max_{\zeta\in[\xi^*-1,\xi^*+1]}\|\v(\zeta;\xi,\omega,\v^*_{\tau})\|^2_{\H}\leq K(\omega,\xi^*,\v^*_{\tau}), \ \text{ for all }\ \xi\in[\xi^*-1,\zeta].
	\end{align}
	Since $\V$ is dense in $\H$, for $\varepsilon>0$ same as in \eqref{RLTT1}, we can find an element $\tilde{\v}\in\V$ such that 	\begin{align}\label{RL}\|\tilde{\v}-\v^*_{\tau}\|_{\H}\leq\frac{\varepsilon^2}{8(K(\omega,\xi^*,\v^*_{\tau})+\|\v^*_{\tau}\|_{\H})}.\end{align}
	Thus, for $\xi\in(\xi^*,\xi^*+\varepsilon^*_1)$ and $\tau\in(\xi,\xi^*+\varepsilon^*_1)$, in view of \eqref{RLTT6} and \eqref{RL}, we obtain
	\begin{align}\label{RLTT7}
		|(\v(\tau;\xi,\omega,\v^*_{\tau})-\v^*_{\tau},\v^*_{\tau})|&\leq|(\v(\tau;\xi,\omega,\v^*_{\tau})-\v^*_{\tau},\tilde{\v})|+|(\v(\tau;\xi,\omega,\v^*_{\tau})-\v^*_{\tau},\tilde{\v}-\v^*_{\tau})|\nonumber\\&\leq|(\v(\tau;\xi,\omega,\v^*_{\tau})-\v^*_{\tau},\tilde{\v})|+\frac{\varepsilon^2}{8}.
	\end{align}
\fi 
	Now, using the first equation of the system \eqref{Div-free}, we have 
	\begin{align}\label{RLTT8}
		&|(\v(\xi;\tau,\omega,\v^*_{\tau})-\v^*_{\tau},{\v^*_{\tau}})|\nonumber\\&=\left|\left\langle\int^{\xi}_{\tau}\frac{\d}{\d\zeta}\v(\zeta;\tau,\omega,\v^*_{\tau})\d\zeta,{\v^*_{\tau}}\right\rangle\right|\nonumber\\&\leq\|{\v^*_{\tau}}\|_{\H}\bigg(\int^{\xi}_{\tau}\bigg\|\frac{\d}{\d\zeta}\v(\zeta;\tau,\omega,\v^*_{\tau})\bigg\|^2_{\H}\d\zeta\bigg)^{1/2}(\tau-\xi)^{1/2}\nonumber\\&\leq C\|{\v^*_{\tau}}\|_{\H}\bigg(\int^{\xi}_{\tau}\bigg[(1+|y(\theta_{\zeta}\omega)|)\|\v(\zeta;\tau,\omega,\v^*_{\tau})\|^2_{\H}+e^{2\sigma y(\theta_{\zeta}\omega)}\|\v(\zeta;\tau,\omega,\v^*_{\tau})\|^2_{\mathbb{L}^{4}(\mathcal{O})}\nonumber\\&\qquad\times\|\nabla\v(\zeta;\tau,\omega,\v^*_{\tau})\|^2_{\L^4(\mathcal{O})}+e^{-2\sigma y(\theta_{\zeta}\omega)}\|\f(\zeta)\|^2_{\H}\bigg]\d\zeta\bigg)^{1/2}(\tau-\xi)^{1/2}\nonumber\\&\leq C\|{\v^*_{\tau}}\|_{\H}\bigg(\int^{\xi}_{\tau}\bigg[(1+|y(\theta_{\zeta}\omega)|)\|\v(\zeta;\tau,\omega,\v^*_{\tau})\|^2_{\H}+e^{2\sigma y(\theta_{\zeta}\omega)}\|\varrho(\zeta;\tau,\omega,\v^*_{\tau})\|^2_{\mathrm{L}^{2}(\mathcal{O})}\nonumber\\&\qquad\times\|\varrho(\zeta;\tau,\omega,\v^*_{\tau})\|^2_{\mathrm{L}^4(\mathcal{O})}+e^{-2\sigma y(\theta_{\zeta}\omega)}\|\f(\zeta)\|^2_{\H}\bigg]\d\zeta\bigg)^{1/2}(\tau-\xi)^{1/2},
	\end{align}
	where we have used \eqref{SE-2D} and \eqref{Grad-Curl1} in the final inequality. Taking the inner product of first equation of \eqref{vorticity} with $|\varrho|^{k-2}\varrho_{\nu}$ with $k\in\{2,4\}$, we obtain (see \eqref{UN12})
\begin{align*}
	\frac{1}{k}\frac{\d}{\d \zeta}\|\varrho(\zeta)\|^k_{\mathrm{L}^k(\mathcal{O})}+\frac{\sigma^2}{2}\|\varrho(\zeta)\|^k_{\mathrm{L}^k(\mathcal{O})}&\leq C e^{-k\sigma y(\theta_{\zeta}\omega)}\|\nabla\land\f(\zeta)\|^k_{\mathrm{L}^k(\mathcal{O})}+(1+\sigma|y(\theta_{\zeta}\omega)|\|\varrho(\zeta)\|^k_{\mathrm{L}^k(\mathcal{O})},
\end{align*}
for a.e. $\zeta\in[0,T]$, which implies by Gronwall's inequality
\begin{align}\label{RLTT41}
	&\sup_{\tau\in[\tau^*-1,\tau^*+1]}\|\varrho(\zeta;\tau,\omega,\v^*_{\tau})\|_{\mathrm{L}^k(\mathcal{O})}^k\nonumber\\&\leq C\bigg[ \|\nabla\land\v^*_{\tau}\|_{\mathrm{L}^k(\mathcal{O})}^k +\|\nabla\land\f\|^k_{\mathrm{L}^{\infty}_{\mathrm{loc}}(\R;\mathrm{L}^k(\mathcal{O}))}\int_{\tau^*-1}^{\zeta}e^{-k\sigma y(\theta_{r}\omega)}\d r\bigg]e^{\int_{\tau^*-1}^{\zeta}(1+\sigma|y(\theta_{r}\omega)|)\d r},
\end{align}
	which is finite and the right hand side is independent of $\tau$. From the fact that  $\v(\cdot;\tau,\omega,\v^*_{\tau})\in\mathrm{C}([\tau,+\infty);\H)$, $\varrho(\cdot;\tau,\omega,\v^*_{\tau})\in\mathrm{L}^{\infty}_{\mathrm{loc}}(\tau,+\infty;\mathrm{L}^{p}(\mathcal{O}))$ for $p\in[2,\infty)$, continuity of $y(\cdot)$, \eqref{RLTT4} and \eqref{RLTT41}, it follows from \eqref{RLTT8} that for $\varepsilon>0$ same as in \eqref{RLTT1}, we can find a positive real number $\varepsilon^*_2=\varepsilon^*_2(\varepsilon,\tau^*,\omega,\v^*_{\tau})$ such that
	\begin{align}\label{RLTT11}
		|(\v(\xi;\tau,\omega,\v^*_{\tau})-\v^*_{\tau},{\v^*_{\tau}})|\leq \frac{\varepsilon^2}{2},  \ \text{ whenever }\  \tau\in(\tau^*,\tau^*+\varepsilon^*_2), \ \xi\in(\tau,\tau^*+\varepsilon^*_2).
	\end{align}
	Taking $\varepsilon^*=\min\{\varepsilon^*_1,\varepsilon^*_2\}$, we obtain \eqref{RLTT1} by combining \eqref{RLTT2}, \eqref{RLTT5} and \eqref{RLTT11}.
	
	Further, one can prove that, for given $\u^*_{\tau}\in\mathcal{W}$, $t\in\R$ and $\omega\in\Omega$, the $\mathcal{W}$-valued mapping $\tau\mapsto S(t,\tau,\omega)\u^*_{\tau}$ is left continuous from $(-\infty,t]$ to $\H$ following the similar steps as above. This completes the proof.
\end{proof}

	Since, $\omega(\cdot)$ has sub-exponential growth  (cf. \cite[Lemma 11]{CGSV}), $\Omega$ can be written as $\Omega=\bigcup\limits_{N\in\N}\Omega_{N}$, where
\begin{align*}
	\Omega_{N}:=\{\omega\in\Omega:|\omega(t)|\leq Ne^{|t|},\text{ for all }t\in\R\}, \text{ for all } \ N\in\N.
\end{align*}
Moreover, for each $N\in\N$, $(\Omega_{N},d_{\Omega_{N}})$ is a polish space (cf. \cite[Lemma 17]{CGSV}).
\begin{lemma}\label{conv_z}
	For each $N\in\N$, suppose $\omega_k,\omega_0\in\Omega_{N}$ such that $d_{\Omega}(\omega_k,\omega_0)\to0$ as $k\to+\infty$. Then, for each $\tau\in\R$, $T\in\R^+$ and $a\in\R$ ,
	\begin{align}
		\sup_{t\in[\tau,\tau+T]}&\bigg[|y(\theta_{t}\omega_k)-y(\theta_{t}\omega_0)|+|e^{a y(\theta_{t}\omega_k)}-e^{a y(\theta_{t}\omega_0)}|\bigg]\to 0 \ \text{ as } \ k\to+\infty,\label{SC3}\\
		\sup_{k\in\N}\sup_{t\in[\tau,\tau+T]}&|y(\theta_{t}\omega_k)|\leq C(\tau,T,\omega_0).\label{conv_z2}
	\end{align}
\end{lemma}
\begin{proof}
	The proof is straightforward (cf. \cite[Corollary 22]{CLL} and \cite[Lemma 2.5]{YR}).
\end{proof}

The following lemma proves that for all $t\in\R$ and $\w\in\mathcal{W}$, the mapping $(\tau,\omega)\mapsto S(t,\tau,\omega)\w$ is measurable from $(\Omega,\mathscr{F})$ to $(\mathcal{W},\mathscr{B}(\mathcal{H})\cap\mathcal{W})$.

 \begin{lemma}\label{H333}
	 	Suppose that $\tau\in\R$, $t\geq\tau$ and $\u_{\tau}\in\mathcal{W}$. For each $N\in\N$, the $\mathcal{W}$-valued mapping $\omega\mapsto S(t,\tau,\omega)\u_{\tau}$ is continuous from $(\Omega_{N},d_{\Omega_N})$ to $\H$.
 \end{lemma}

\begin{proof}
	\iffalse 
		Let $\omega_k,\omega_0\in\Omega_N,\ N\in\mathbb{N}$ such that $d_{\Omega_N}(\omega_k,\omega_0)\to0$ as $k\to+\infty$. It implies that for $a\in\R$
		\begin{align}
			\lim\limits_{k\to+\infty}|z^a(t,\omega_k)-z^a(t,\omega_0)|=0 \ \text{ and } \	\sup_{k\in\N}\sup_{t\in[\tau,\tau+T]}|z(t,\omega_k)|\leq C(\tau,T,\omega_0).
		\end{align}
	\fi 
	 Let $\mathbf{V}^k(\cdot):=\v^k(\cdot)-\v^0(\cdot),$ where $\v^k(\cdot):=\v(\cdot;\tau,\omega_k,\v_{\tau})$ and $\v^0(\cdot):=\v(\cdot;\tau,\omega_0,\v_{\tau})$. Then, $\mathbf{V}^k(\cdot)$ satisfies:
\begin{align*}
		&\frac{\d\mathbf{V}^k}{\d r}+\frac{\sigma^2}{2}\mathbf{V}^k\nonumber\\&=-e^{\sigma y(\theta_{r}\omega_0)}\big[\B(\v^k,\mathbf{V}^k)+\B(\mathbf{V}^k,\v^0)\big]-\left[e^{\sigma y(\theta_{r}\omega_k)}-e^{\sigma y(\theta_{r}\omega_0)}\right]\big[\B(\v^k,\mathbf{V}^k)+\B(\v^k,\v^0)\big]\nonumber\\&\quad+[e^{-\sigma y(\theta_{r}\omega_k)}-e^{-\sigma y(\theta_{r}\omega_0)}]\boldsymbol{f}+\sigma \big[y(\theta_{r}\omega_k)-y(\theta_{r}\omega_0)\big]\v^0+\sigma y(\theta_{r}\omega_k)\mathbf{V}^k,
\end{align*}
in $\H$. By taking the inner product with $\mathbf{V}^k(\cdot)$ in the above equation, we obtain (for $p>3$)
\begin{align*}
&\frac{1}{2}\frac{\d}{\d r}\|\mathbf{V}^k(r)\|^2_{\H}+\frac{\sigma^2}{2}\|\mathbf{V}^k(r)\|^2_{\H}\nonumber\\&=-e^{\sigma y(\theta_{r}\omega_0)}b(\mathbf{V}^k(r),\v^0(r),\mathbf{V}^k(r))-\left[e^{\sigma y(\theta_{r}\omega_k)}-e^{\sigma y(\theta_{r}\omega_0)}\right]b(\v^{k}(r),\v^0(r),\mathbf{V}^k(r))\nonumber\\&\quad+[e^{-\sigma y(\theta_{r}\omega_k)}-e^{-\sigma y(\theta_{r}\omega_0)}](\boldsymbol{f}(r),\mathbf{V}^k(r))+\sigma \big[y(\theta_{r}\omega_k)-y(\theta_{r}\omega_0)\big](\v^0(r),\mathbf{V}^k(r))\nonumber\\&\quad+\sigma y(\theta_{r}\omega_k)\|\mathbf{V}^k(r)\|^2_{\H}\nonumber\\&\leq e^{\sigma y(\theta_{r}\omega_0)}\|\nabla\v^0(r)\|_{\L^p(\mathcal{O})}\|\mathbf{V}^k(r)\|^{2}_{\L^{\frac{2p}{p-1}}(\mathcal{O})}+|e^{-\sigma y(\theta_{r}\omega_k)}-e^{-\sigma y(\theta_{r}\omega_0)}|\|\f(r)\|_{\H}\|\mathbf{V}^k(r)\|_{\H}\nonumber\\&\quad+\left|e^{\sigma y(\theta_{r}\omega_k)}-e^{\sigma y(\theta_{r}\omega_0)}\right|\|\v^{k}(r)\|_{\L^4(\mathcal{O})}\|\nabla\v^0(r)\|_{\L^4(\mathcal{O})}\|\mathbf{V}^k(r)\|_{\H}\nonumber\\&\quad+\sigma |y(\theta_{r}\omega_k)-y(\theta_{r}\omega_0)|\|\v^0(r)\|_{\H}\|\mathbf{V}^k(r)\|_{\H}+\sigma |y(\theta_{r}\omega_k)|\|\mathbf{V}^k(r)\|^2_{\H}\nonumber\\&\leq e^{\sigma y(\theta_{r}\omega_0)}\|\nabla\v^0(r)\|_{\L^p(\mathcal{O})}\|\mathbf{V}^k(r)\|^{\frac{2}{p-2}}_{\L^{p}(\mathcal{O})}\|\mathbf{V}^k(r)\|^{\frac{2(p-3)}{p-2}}_{\H}+C|e^{-\sigma y(\theta_{r}\omega_k)}-e^{-\sigma y(\theta_{r}\omega_0)}|^2\|\f(r)\|^2_{\H}\nonumber\\&\quad+C\left|e^{\sigma y(\theta_{r}\omega_k)}-e^{\sigma y(\theta_{r}\omega_0)}\right|^2\|\v^{k}(r)\|^2_{\L^4(\mathcal{O})}\|\nabla\v^0(r)\|^2_{\L^4(\mathcal{O})}+C|y(\theta_{r}\omega_k)-y(\theta_{r}\omega_0)|^2\|\v^0(r)\|^2_{\H}\nonumber\\&\quad+\sigma |y(\theta_{r}\omega_k)|\|\mathbf{V}^k(r)\|^2_{\H}+\frac{\sigma^2}{4}\|\mathbf{V}^k(r)\|_{\H}^2\nonumber\\&\leq (Cp)^{\frac{1}{p-2}}\frac{Cp^{\frac{1}{p}}p^2}{p-1}\|\mathbf{V}^k(r)\|^{\frac{2(p-3)}{p-2}}_{\H}+C\left|e^{\sigma y(\theta_{r}\omega_k)}-e^{\sigma y(\theta_{r}\omega_0)}\right|^2+C|y(\theta_{r}\omega_k)-y(\theta_{r}\omega_0)|^2\nonumber\\&\quad+C|e^{-\sigma y(\theta_{r}\omega_k)}-e^{-\sigma y(\theta_{r}\omega_0)}|^2\|\f(r)\|^2_{\H}+\sigma |y(\theta_{r}\omega_k)|\|\mathbf{V}^k(r)\|^2_{\H}+\frac{\sigma^2}{4}\|\mathbf{V}^k(r)\|_{\H}^2,
\end{align*}
for a.e. $r\in[\tau,\tau+T]$, where we have used the bounds from \eqref{UN3} and \eqref{SC3}, and the constant 
\begin{align*}
	C=C(T,\omega_0,\mathcal{O},\|\u_{\tau}\|_{\V},\|\f\|_{\mathrm{L}^{2}(\tau,\tau+T;\V)},\|\nabla\land\u_{\tau}\|_{\mathrm{L}^{\infty}(\mathcal{O})},\|\nabla\land\f\|_{\mathrm{L}^{\infty}(\tau,\tau+T;\mathrm{L}^{\infty}(\mathcal{O}))})>0,
\end{align*}
 is independent of $p$ and $k$. Taking integration from $\tau$ to $t>\tau$, we obtain 
\begin{align}
	\|\mathbf{V}^k(t)\|^2_{\H}&\leq |e^{-2\sigma y(\theta_{\tau}\omega_k)}-e^{-2\sigma y(\theta_{\tau}\omega_0)}|\|\u_{\tau}\|^2_{\H}+ C\int_{\tau}^{t}|e^{-\sigma y(\theta_{r}\omega_k)}-e^{-\sigma y(\theta_{r}\omega_0)}|^2\|\f(r)\|^2_{\H}\d r\nonumber\\&\quad+C\int_{\tau}^{t}\left|e^{\sigma y(\theta_{r}\omega_k)}-e^{\sigma y(\theta_{r}\omega_0)}\right|^2\d r +(Cp)^{\frac{1}{p-2}}\frac{Cp^{\frac{1}{p}}p^2}{p-1}\int_{\tau}^{t}\|\mathbf{V}^k(r)\|^{\frac{2(p-3)}{p-2}}_{\H}\d r\nonumber\\&\quad+2\sigma\int_{\tau}^{t}|y(\theta_{r}\omega)|\|\mathbf{V}^k(r)\|^2_{\H}\d r\label{SC1}\\&=: M(k)+ (Cp)^{\frac{1}{p-2}}\frac{Cp^{\frac{1}{p}}p^2}{p-1}\int_{\tau}^{t}\|\mathbf{V}^k(r)\|^{\frac{2(p-3)}{p-2}}_{\H}\d r+2\sigma\int_{\tau}^{t}|y(\theta_{r}\omega)|\|\mathbf{V}^k(r)\|^2_{\H}\d r,\nonumber
\end{align}
where $M(k)$ is the first three terms appearing in the right hand side of \eqref{SC1}. By applying \cite[Theorem 21]{Dragomir} to \eqref{SC1}, we arrive at
\begin{align*}
	&\|\mathbf{V}^k(t)\|^2_{\H}\leq \bigg[\{M(k)\}^{\frac{1}{p-2}}+ \frac{(Cp)^{\frac{1}{p-2}}Cp^{\frac{1}{p}}p^2}{(p-2)(p-1)} (t-\tau)\bigg]^{p-2}e^{2\sigma\int_{\tau}^{t}|y(\theta_{r}\omega)|\d r},
\end{align*}
for all $t\in[\tau,\tau+T]$. Now, using the convergence in \eqref{SC3}, we find 
\begin{align*}
	&\lim_{k\to+\infty}\|\mathbf{V}^k(t)\|^2_{\H}\leq  (Cp)^{2}\{C(t-\tau)\}^{p-2}\frac{1}{p^{\frac{2}{p}}}\bigg[\frac{p^{2}}{(p-2)(p-1)}\bigg]^{p-2}e^{2\sigma\int_{\tau}^{t}|y(\theta_{r}\omega)|\d r}.
\end{align*}
Let us now take $t$ sufficiently small such that $C(t-\tau)<1$. Further, it provides us that $\lim\limits_{p\to+\infty}[C(t-\tau)]^{p-2}(Cp)^2e^{2\sigma\int_{\tau}^{t}|y(\theta_{r}\omega)|\d r}=0,$ whereas $\lim\limits_{p\to+\infty}p^{-\frac{2}{p}}\left[\frac{p^2}{(p-1)(p-2)}\right]^{p-2}=e^3$. Hence $$\lim\limits_{k\to+\infty}\|\mathbf{V}^k(t)\|^2_{\H}=0 \ \  \mbox{ if }\ C(t-\tau)<1.$$ We repeat the same arguments a finite number of times to cover the  interval $[\tau,\tau+T]$, which completes the proof.
\end{proof}

\begin{lemma}\label{Absorbing}
	Suppose that Hypothesis \ref{Hypo-f-3} holds. Then, there exists a $\d_{\mathcal{W}}$-bounded absorbing set for each $t\in\R$ which is $\delta$-compact.
\end{lemma}
\begin{proof}
	From \eqref{AE} and \eqref{UNS}, we infer
	\begin{align}\label{AB1}
			&\|\u(t;\tau,\omega,\u_{\tau})\|^2_{\V}%\nonumber\\&\leq C e^{2\sigma y(\theta_{t}\omega)}\bigg[e^{\int_{\tau}^{t}\left[-\frac{\sigma^2}{2}+2\sigma y(\theta_{r}\omega)\right]\d r-2\sigma y(\theta_{\tau}\omega)} \|\u_{\tau}\|_{\V}^2\nonumber\\&\quad+\int_{\tau}^{t}e^{\int_{\xi}^{t}\left[-\frac{\sigma^2}{2}+2\sigma y(\theta_{r}\omega)\right]\d r-2\sigma y(\theta_{\xi}\omega)}			\big\{\|\f(\xi)\|^2_{\L^2(\mathcal{O})}+\|\nabla\land\boldsymbol{f}(\xi)\|_{\mathrm{L}^2(\mathcal{O})}^2\big\}\d\xi\bigg]
			\nonumber\\&\leq C e^{-\frac{\sigma^2}{2}t+2\sigma\int_{0}^{t} y(\theta_{r}\omega)\d r+2\sigma y(\theta_{t}\omega)}\bigg[e^{\frac{\sigma^2}{2}\tau+2\sigma\int_{\tau}^{0} y(\theta_{r}\omega)\d r-2\sigma y(\theta_{\tau}\omega)} \|\u_{\tau}\|_{\V}^2\nonumber\\&\quad+\int_{\tau}^{t}e^{\frac{\sigma^2}{2}\xi+2\sigma\int_{\xi}^{0} y(\theta_{r}\omega)\d r-2\sigma y(\theta_{\xi}\omega)}
			\big\{\|\f(\xi)\|^2_{\L^2(\mathcal{O})}+\|\nabla\land\boldsymbol{f}(\xi)\|_{\mathrm{L}^2(\mathcal{O})}^2\big\}\d\xi\bigg],
	\end{align}
	\iffalse 
	\begin{align}
			\|\u(t;\tau,\omega,\u_{\tau})\|^2_{\V}\leq \frac{C}{z^2(t,\omega)}\bigg[z^2(\tau,\omega)\|\u_{\tau}\|_{\V}^2e^{-\frac{\sigma^2}{2}(t-\tau)}+\int_{\tau}^{t}e^{-\frac{\sigma^2}{2}(t-\xi)} z^2(\xi,\omega)\|\f(\xi)\|^2_{\V}\d\xi\bigg],
	\end{align}
\fi 
and 
\begin{align}\label{AB2}
	&\|\nabla\land\u(t;\tau,\omega,\u_{\tau})\|_{\L^{\infty}(\mathcal{O})}\nonumber\\&\leq e^{-\frac{\sigma^2}{2}t+2\sigma\int_{0}^{t} y(\theta_{r}\omega)\d r+\sigma y(\theta_{t}\omega)}\bigg[e^{\frac{\sigma^2}{2}\tau+2\sigma\int_{\tau}^{0} y(\theta_{r}\omega)\d r-\sigma y(\theta_{\tau}\omega)}\|\nabla\land\u_{\tau}\|_{\mathrm{L}^{\infty}(\mathcal{O})}\nonumber\\&\quad+\int_{\tau}^{t}e^{\frac{\sigma^2}{2}\xi+2\sigma\int_{\xi}^{0} y(\theta_{r}\omega)\d r-\sigma y(\theta_{\xi}\omega)}\|\nabla\land\f(\xi)\|_{\mathrm{L}^{\infty}(\mathcal{O})}\d\xi\bigg].
\end{align}
\iffalse 
	\begin{align}
	&\|\nabla\land\u(t;\tau,\omega,\u_{\tau})\|_{\mathrm{L}^{\infty}(\mathcal{O})}\nonumber\\&\leq\frac{1}{z(t,\omega)}\bigg[e^{-\frac{\sigma^2}{2}(t-\tau)}z(\tau,\omega)\|\nabla\land\u_{\tau}\|_{\mathrm{L}^{\infty}(\mathcal{O})}+\int_{\tau}^{t}e^{-\frac{\sigma^2}{2}(t-r)}z(r,\omega)\|\nabla\land\f(r)\|_{\mathrm{L}^{\infty}(\mathcal{O})}\d r\bigg].
\end{align}
\fi 
%Let $\delta:=\max\{\delta_1,\delta_2\}$, where $\delta_1$ and $\delta_2$ both are same as in \eqref{forcing1}. 
From \eqref{Z3}, we obtain  the existence of $R_1, R_2, R_3, R_4<0$ such that 
\begin{align*}
	-2\sigma y(\theta_{\xi}\omega)&\leq-\left(\frac{\sigma^2}{4}-\frac{\delta_1}{2}\right)\xi, \ \ \ \ \ \ \  \mbox{for all $\xi\leq R_1,$}\\
	-\sigma y(\theta_{\xi}\omega)&\leq-\left(\frac{\sigma^2}{4}-\frac{\delta_2}{2}\right)\xi, \ \ \ \ \ \ \  \mbox{for all $\xi\leq R_2,$}\\
	\frac{\sigma^2}{2}\xi+2\sigma\int^{0}_{\xi} y(\theta_{r}\omega)\d r&\leq \left(\frac{\sigma^2}{4}+\frac{\delta_1}{2}\right)\xi, \ \ \ \ \ \ \ \ \mbox{for all $\xi\leq R_3,$}\\
		\frac{\sigma^2}{2}\xi+2\sigma\int^{0}_{\xi} y(\theta_{r}\omega)\d r&\leq \left(\frac{\sigma^2}{4}+\frac{\delta_2}{2}\right)\xi, \ \ \ \ \ \ \ \ \mbox{for all $\xi\leq R_4,$}
\end{align*}
 where $\sigma$ is the constant appearing  in \eqref{1}, and  $\delta_1$ and $\delta_2$ are the constants appearing  in \eqref{G3}. Therefore, for all $\xi\leq R:=\min\{R_1,R_2,R_3,R_4\}$,
\begin{align}\label{y-bound1}
		\frac{\sigma^2}{2}\xi+2\sigma\int^{0}_{\xi} y(\theta_{r}\omega)\d r-2\sigma y(\theta_{\xi}\omega)&\leq \delta_1\xi,
\end{align}
and 
\begin{align}\label{y-bound2}
	\frac{\sigma^2}{2}\xi+2\sigma\int^{0}_{\xi} y(\theta_{r}\omega)\d r-\sigma y(\theta_{\xi}\omega)&\leq \delta_2\xi.
\end{align}
From \eqref{y-bound1}-\eqref{y-bound2} and Hypothesis \ref{Hypo-f-3}, we infer that the right hand sides of \eqref{AB1} and \eqref{AB2} are bounded as $\tau\to-\infty$. Therefore, for given bounded subset $B\subset\mathcal{W}$, $t\in\R$ and $\omega\in\Omega$, there exist a $\d_{\mathcal{W}}$-bounded set $\mathcal{B}(t,\omega)\subset\mathcal{W}$, and a time $\tau_0(B)$ depending  on $B$ such that 
\begin{align*}
	S(t,\tau,\omega)B\subset\mathcal{B}(t,\omega),  \ \ \text{ for all } \ \tau\leq \tau_0.
\end{align*}
Since the embedding $\V\subset\H$ is compact, we deduce that there exists a $\d_{\mathcal{W}}$-bounded and $\delta$-compact random ball which absorbs $\d_{\mathcal{W}}$-bounded sets of $\mathcal{W},$ for all $t\in\R$ and $\omega\in\Omega$, and this completes the proof.
\end{proof}
The following theorem demonstrates the main result of this section which is a direct consequence of the abstract result (Theorem \ref{SWA}).

\begin{theorem}\label{M-Thm1}
	Suppose that Hypotheses \ref{Hypo-f-3}, \ref{H1} and \ref{H2} hold. Then, there exists a \textsl{stochastic weak attractor} $\mathscr{A}(t,\omega)$ for the system \eqref{1} for all $t\in\R$ and $\omega\in\Omega$, which attracts bounded sets from $-\infty$.
\end{theorem}
\begin{proof}
	Lemma \ref{H111} ensures  that Hypothesis \ref{H1} is satisfied, Lemmas \ref{H222} and \ref{H333} verify  Hypothesis \ref{H2}, and Lemma \ref{Absorbing} proves the existence of $\d_{\mathcal{W}}$-bounded and $\delta$-compact random absorbing set. Therefore, by Theorem \ref{SWA}, there exists a stochastic weak attractor for the system \eqref{1} for all $t\in\R$ and $\omega\in\Omega$, which attracts $\d_{\mathcal{W}}$-bounded sets from $-\infty$. Hence the proof is completed. 
\end{proof}

\subsection{Existence of stochastic weak attractor}
Let us now consider the following  autonomous 2D stochastic EE in $\mathcal{O}$:
\begin{equation}\label{11}
	\left\{
	\begin{aligned}
		\d \bar{\u}+\left[(\bar{\u}\cdot\nabla)\bar{\u}+\nabla \bar{p}\right]\d t&=\boldsymbol{f}_{\infty}\d t+ \sigma\bar{\u}\d\mathrm{W}, \ \text{ in } \ \mathcal{O}\times(\tau,\infty), \\ \nabla\cdot\bar{\u}&=0, \hspace{26mm} \text{ in } \ \mathcal{O}\times[0,\infty), \\
		\bar{\u}\cdot\boldsymbol{n} &=0, \hspace{26mm} \text{ on } \ \partial\mathcal{O}\times[0,\infty), \\
		\bar{\u}|_{t=0}&=\bar{\u}_{0}, \hspace{24mm} \text{ in } \ \mathcal{O},\\			
	\end{aligned}
	\right.
\end{equation}
where $\f_{\infty}$ is a time-independent external forcing. Define a new variable $\bar{\v}$  by 
\begin{align}\label{COV-Auto}
	\bar{\v}(t;\omega,\v_{0})=e^{-\sigma y(\theta_{t}\omega)}\bar{\u}(t;\omega,\u_{0}) \ \ \
	\text{ with }
	\ \ \	\bar{\v}_{0}=e^{-\sigma y(\omega)}\bar{\u}_{0},
\end{align}
where $\bar{\u}(t;\omega,\u_{0})$ and $y(\theta_{t}\omega)$ are the solutions of \eqref{11} and \eqref{OU2}, respectively. Then $\bar{\v}(\cdot):=\bar{\v}(\cdot;\omega,\v_{0})$ satisfies:
\begin{equation}\label{EuEq-Auto}
	\left\{
	\begin{aligned}
		\frac{\d\bar{\v}}{\d t}+\left[\frac{\sigma^2}{2}-\sigma y(\theta_{t}\omega)\right]\bar{\v}+e^{\sigma y(\theta_{t}\omega)}(\bar{\v}\cdot\nabla)\bar{\v}+e^{-\sigma y(\theta_{t}\omega)}\nabla \bar{p}&=e^{-\sigma y(\theta_{t}\omega)}\f_{\infty} ,\hspace{5mm} \text{ in } \ \mathcal{O}\times(\tau,\infty), \\ \nabla\cdot\bar{\v}&=0, \hspace{17mm} \text{ in } \ \mathcal{O}\times(\tau,\infty), \\
		\bar{\v}\cdot\boldsymbol{n} &=0,\hspace{17mm} \text{ on } \ \partial\mathcal{O}\times(\tau,\infty), \\
		\bar{\v}|_{t=0}&=\bar{\v}_{0}, \hspace{15mm} \text{ in } \mathcal{O},
	\end{aligned}
	\right.
\end{equation}
in the distributional sense, where $\sigma>0$. We have shown in Theorem \ref{M-Thm1} that there exists a pullback stochastic weak attractor $\mathscr{A}(t,\omega)$ for the system \eqref{1} with time-dependent forcing $\f$. For the time-independent forcing  $\f_{\infty}\in\L^2(\mathcal{O})$ and  $\nabla\land\f_{\infty}\in\mathrm{L}^{\infty}(\mathcal{O}),$ the Hypothesis \ref{Hypo-f-3} is automatically satisfied for $\delta_1,  \delta_2\in(0,\frac{\sigma^2}{2})$ . Since the sample space is the two-sided Wiener space $\C_0(\R;\R)$, Hypothesis \ref{H3} is satisfied.   In view of the abstract result stated in Theorem \ref{SWA1}, the proof of the following Theorem is analogues to the proof of Theorem \ref{M-Thm1}. Therefore, we are only state the theorem and not repeat the proof here. 
	\begin{theorem}\label{M-Thm2}
		Suppose that Hypotheses \ref{H1} and \ref{H2} hold. Let the time-independent forcing $\f_{\infty}$ be such that $\f_{\infty}\in\L^2(\mathcal{O})$ and $\nabla\land\f_{\infty}\in\mathrm{L}^{\infty}(\mathcal{O})$. Then, there exists a \textsl{stochastic weak attractor} $\mathscr{A}_{\infty}(\omega)$ for the system \eqref{11} for all $\omega\in\Omega$, which attracts bounded sets from $-\infty$.
	\end{theorem}

\section{Weak asymptotic autonomy of pullback stochastic weak attractors}\label{sec5}\setcounter{equation}{0}

In this section, we apply the abstract results established in Section \ref{AT1} to the 2D non-autonomous stochastic Euler system \eqref{1}. In Section \ref{sec4}, we proved the existence of a minimal pullback stochastic weak attractor (Theorem \ref{SWA}) for the system \eqref{1} and the existence of a minimal stochastic weak attractor (Theorem \ref{SWA1}) for the system \eqref{11}. In order to apply  Theorem \ref{WAA-MT}, we have to show that $\{\mathscr{A}(\tau,\omega)\}_{\tau\in\R,\omega\in\Omega}$ and $\{\mathscr{A}_{\infty}(\omega)\}_{\omega\in\Omega}$ satisfy the all three conditions stated in Theorem \ref{WAA-MT}. %For that purpose, we need the following assumptions on the non-autonomous forcing $\f(\cdot)$ and autonomous forcing $\f_{\infty}$.

The following lemma provides  us that $\{\mathscr{A}(\tau,\omega)\}_{\tau\in\R,\omega\in\Omega}$  satisfies condition $(\mathrm{i})$  of Theorem \ref{WAA-MT}. 

\begin{lemma}\label{(i)}
	$\overline{\bigcup\limits_{t\geq\tau}\mathscr{A}(t,\omega)}^{\delta}$ is $\d_{\mathcal{W}}$-bounded and $\delta$-compact. 
\end{lemma}

\begin{proof}
	 Let $\{\w_n\}_{n=1}^{\infty}$ be an arbitrary sequence extracted from $\bigcup\limits_{t\geq\tau}\mathscr{A}(t,\omega)$. Then, there exists  a sequence $t_n\geq\tau$ such that $\w_n\in\mathscr{A}(t_n,\omega)$ for each $n\in\N$. Now, for the sequence $s_n\to+\infty$, by the invariance property of $\mathscr{A},$ we have $\w_n\in S(t_n,t_n-s_n,\omega)\mathscr{A}(t_n-s_n,\theta_{-s_n}\omega)$. It implies that we can find $\w_{0,n}\in\mathscr{A}(t_n-s_n,\theta_{-s_n}\omega)$ such that $\w_n=S(t_n,t_n-s_n,
	\omega)\w_{0,n}$. By the property of absorbing set, there exists a number $N_0\in\N$ such that $\w_{0,n}\in\mathscr{A}(t_n-s_n,\theta_{-s_n}\omega)\subseteq\mathcal{B}(t_n-s_n,\theta_{-s_n}\omega)$ for all $n\geq N_0$. Since the absorbing set is $\d_{\mathcal{W}}$-bounded and $\delta$-compact, the sequence  $\{\w_{0,n}\}_{n=1}^{\infty}$ is $\d_{\mathcal{W}}$-bounded. It implies that $\{S(t_n,t_n-s_n,
	\omega)\w_{0,n}\}_{n=1}^{\infty}$ or $\{\w_n\}_{n=1}^{\infty}$ is also a $\d_{\mathcal{W}}$-bounded sequence. We infer from $\d_{\mathcal{W}}$-bounded sequence $\{\w_n\}_{n=1}^{\infty}$ that the sequences $\{\w_n\}_{n=1}^{\infty}$ and $\{\nabla\land\w_n\}_{n=1}^{\infty}$ are bounded in $\V$ and $\mathrm{L}^{\infty}(\mathcal{O})$, respectively. In view of the Banach-Alaoglu theorem, and the uniqueness of weak and weak$^*$ limits, we find an element $\w\in\mathcal{W}$ such that $\w_n\rightharpoonup \w$ weakly in $\V$ and $\nabla\land\w_n\xrightharpoonup{*} \nabla\land\w$ weak$^*$ in $\mathrm{L}^{\infty}(\mathcal{O})$. Now, the compact embedding $\V\hookrightarrow\H$ gives that $\w_n\to\w$ strongly in $\H$. Hence, the arbitrary sequence $\{\w_n\}_{n=1}^{\infty}$ converges strongly to an element $\w\in\mathcal{W}$ in $\delta$-metric. This completes the proof. 
\end{proof}

The following lemma shows that the stochastic dynamical systems $S(\cdot,\cdot,\cdot)$ and $T(\cdot,\cdot,\cdot)$ satisfy condition $(\mathrm{ii})$  of the Theorem \ref{WAA-MT}. 
\begin{proposition}\label{For_conver-N}
	Suppose that Hypothesis \ref{Hypo_f-N} is satisfied. Then the solution $\v$ of the system \eqref{EuEq} forward converges to the solution $\bar{\v}$ of the system \eqref{EuEq-Auto} in $\delta$-metric, that is,
	\begin{align*}
		\lim_{\tau\to +\infty}\delta(\v(T+\tau,\tau,\theta_{-\tau}\omega,\v_{\tau}),\bar{\v}(T,\omega,\bar{\v}_0))=0, \ \ \text{ for all } T>0 \ \text{ and }\  \omega\in\Omega,
	\end{align*}
	whenever $\{\v_{\tau}\}_{\tau\in\R}$ is a $\d_{\mathcal{W}}$-bounded sequence, $\bar{\v}_{0}\in\mathcal{W}$ and $\delta(\v_{\tau},\bar{\v}_0)\to0$ as $\tau\to+\infty.$
\end{proposition}
\begin{proof}
	Let $\mathscr{V}_{\tau}(t):=\v(t+\tau,\tau,\theta_{-\tau}\omega,\v_{\tau})-\bar{\v}(t,\omega,\bar{\v}_0)$ for $t\geq0$. Then, $\mathscr{V}_{\tau}(\cdot)$ satisfies for a.e. $t\geq0$ and $\boldsymbol{\phi}\in\V$
	\begin{align*}
		&\left\langle\frac{\d\mathscr{V}_{\tau}}{\d t},\boldsymbol{\phi}\right\rangle+\left[\frac{\sigma^2}{2}-\sigma y(\theta_{t}\omega)\right](\mathscr{V}_{\tau},\boldsymbol{\phi})\nonumber\\&=-e^{\sigma y(\theta_{t}\omega)}\langle(\mathscr{V}_{\tau}\cdot\nabla)\bar{\v},\boldsymbol{\phi}\rangle-e^{\sigma y(\theta_{t}\omega)}\langle(\v\cdot\nabla)\mathscr{V}_{\tau},\boldsymbol{\phi}\rangle+e^{-\sigma y(\theta_{t}\omega)}(\f(t+\tau)-\f_{\infty},\boldsymbol{\phi}).
	\end{align*}  
Taking $\boldsymbol{\phi}=\mathscr{V}_{\tau}$, we obtain
\begin{align}\label{FC1}
	&\frac{\d}{\d t}\|\mathscr{V}_{\tau}(t)\|^2_{\H}\nonumber\\&=-\left[\sigma^2-2\sigma y(\theta_{t}\omega)\right]\|\mathscr{V}_{\tau}(t)\|^2_{\H}-2e^{\sigma y(\theta_{t}\omega)}\langle(\mathscr{V}_{\tau}(t)\cdot\nabla)\bar{\v}(t),\mathscr{V}_{\tau}(t)\rangle\nonumber\\&\quad+2e^{-\sigma y(\theta_{t}\omega)}(\f(t+\tau)-\f_{\infty},\mathscr{V}_{\tau}(t))\nonumber\\&\leq  2e^{\sigma y(\theta_{t}\omega)} \|\nabla\bar{\v}(t)\|_{\L^p(\mathcal{O})}\|\mathscr{V}_{\tau}(t)\|^2_{\L^{\frac{2p}{p-1}}(\mathcal{O})}+2\sigma|y(\theta_{t}\omega)|\|\mathscr{V}_{\tau}(t)\|^2_{\H}\nonumber\\&\quad+\frac{e^{-2\sigma y(\theta_{t}\omega)}}{\sigma^2}\|\f(t+\tau)-\f_{\infty}\|^2_{\H}\nonumber\\&\leq 2e^{\sigma y(\theta_{t}\omega)} \|\nabla\bar{\v}(t)\|_{\L^p(\mathcal{O})}\|\mathscr{V}_{\tau}(t)\|^{\frac{2}{p-2}}_{\L^{p}(\mathcal{O})}\|\mathscr{V}_{\tau}(t)\|^{\frac{2(p-3)}{p-2}}_{\H}+2\sigma|y(\theta_{t}\omega)|\|\mathscr{V}_{\tau}(t)\|^2_{\H}\nonumber\\&\quad+\frac{e^{-2\sigma y(\theta_{t}\omega)}}{\sigma^2}\|\f(t+\tau)-\f_{\infty}\|^2_{\H}\nonumber\\&\leq (Cp)^{\frac{1}{p-2}}\frac{Cp^{\frac{1}{p}}p^2}{p-1}e^{\sigma y(\theta_{t}\omega)}\|\mathscr{V}_{\tau}(t)\|^{\frac{2(p-3)}{p-2}}_{\H}+2\sigma|y(\theta_{t}\omega)|\|\mathscr{V}_{\tau}(t)\|^2_{\H}+\frac{e^{-2\sigma y(\theta_{t}\omega)}}{\sigma^2}\|\f(t+\tau)-\f_{\infty}\|^2_{\H},
\end{align}
for a.e. $t\geq0$, where we have used the bounds from \eqref{UN3}, and the constant $C$ is independent of $p$ and $\tau$. By applying \cite[Theorem 21]{Dragomir} to \eqref{FC1}, we arrive at
\begin{align}
	\|\mathscr{V}_{\tau}(T)\|^2_{\H}\leq\biggl\{[\widehat{M}(\tau)]^{\frac{1}{p-2}}+\frac{(Cp)^{\frac{1}{p-2}}Cp^{\frac{1}{p}}p^2}{(p-1)(p-2)}\int_{0}^{T}e^{\sigma y(\theta_{t}\omega)}\d t\biggr\}^{p-2}e^{2\sigma\int_{0}^{T}|y(\theta_{r}\omega)|\d r},
\end{align}
where
\begin{align*}
	\widehat{M}(\tau):=\|\v_{\tau}-\bar{\v}_0\|^2_{\H}+\int_{0}^{T} \frac{e^{-2\sigma y(\theta_{t}\omega)}}{\sigma^2}\|\f(t+\tau)-\f_{\infty}\|^2_{\H}.
\end{align*}
By Hypothesis \ref{Hypo_f-N} and $\lim\limits_{\tau\to+\infty}\|\v_{\tau}-\bar{\v}_0\|_{\H}=0$, we have $\lim\limits_{\tau\to+\infty}\widehat{M}(\tau)=0$. This fact gives 
\begin{align}
	\lim_{\tau\to +\infty}\|\mathscr{V}_{\tau}(T)\|^2_{\H}\leq Cp^2p^{-\frac{2}{p}}\bigg[\frac{p^2}{(p-1)(p-2)}\bigg]^{p-2}\bigg[C\int_{0}^{T}e^{\sigma y(\theta_{t}\omega)}\d t\bigg]^{p-2}e^{2\sigma\int_{0}^{T}|y(\theta_{r}\omega)|\d r}.
\end{align}
Let us now take $T$ sufficiently small such that $C\int_{0}^{T}e^{\sigma y(\theta_{t}\omega)}\d t<1$. Further, it provides us that $\lim\limits_{p\to+\infty}p^2[C\int_{0}^{T}e^{\sigma y(\theta_{t}\omega)}\d t]^{p-2}e^{2\sigma\int_{0}^{T}|y(\theta_{r}\omega)|\d r}=0,$ whereas $\lim\limits_{p\to+\infty}p^{-\frac{2}{p}}\left[\frac{p^2}{(p-1)(p-2)}\right]^{p-2}=e^3$. Hence $\lim\limits_{\tau\to+\infty}\|\mathscr{V}_{\tau}(T)\|^2_{\H}=0$ if $C\int_{0}^{T}e^{\sigma y(\theta_{t}\omega)}\d t<1$. We repeat the same arguments a finite number of times to complete the proof.
\end{proof}

The following lemma demonstrates that the property $\mathrm{(iii)}$ of Theorem \ref{WAA-MT} hold.

\begin{lemma}\label{(iii)}
	There exists a $t_0>0$ such that $\mathcal{B}_{t_0}(\omega):=\overline{\bigcup\limits_{t\geq t_0}\mathcal{B}(t, \omega)}^{\delta}$ is a $\d_{\mathcal{W}}$-bounded set. 
\end{lemma}
\begin{proof}
	It is clear from Lemma \ref{Absorbing} that 
	\begin{align*}
		\mathcal{B}(t,\omega)=\{\u\in\mathcal{W}: \|\u\|_{\V}\leq L_1(t,\omega) \ \text{ and }\ \|\nabla\land\u\|_{\mathrm{L}^{\infty}(\Omega)}\leq L_2(t,\omega) \},
	\end{align*}
where 
\begin{align*}
	L_1(t,\omega)&:= 2 e^{-\frac{\sigma^2}{2}t+2\sigma\int_{0}^{t} y(\theta_{r}\omega)\d r+2\sigma y(\theta_{t}\omega)}\int_{-\infty}^{t}e^{\frac{\sigma^2}{2}\xi+2\sigma\int_{\xi}^{0} y(\theta_{r}\omega)\d r-2\sigma y(\theta_{\xi}\omega)}
	\big\{\|\f(\xi)\|^2_{\L^2(\mathcal{O})}\nonumber\\&\qquad+\|\nabla\land\boldsymbol{f}(\xi)\|_{\mathrm{L}^2(\mathcal{O})}^2\big\}\d\xi,
\end{align*}
and 
\begin{align*}
	L_2(t,\omega)&:=2 e^{-\frac{\sigma^2}{2}t+2\sigma\int_{0}^{t} y(\theta_{r}\omega)\d r+\sigma y(\theta_{t}\omega)}\int_{-\infty}^{t}e^{\frac{\sigma^2}{2}\xi+2\sigma\int_{\xi}^{0} y(\theta_{r}\omega)\d r-\sigma y(\theta_{\xi}\omega)}\|\nabla\land\f(\xi)\|_{\mathrm{L}^{\infty}(\mathcal{O})}\d\xi.
\end{align*}
By the properties of $y(\cdot)$ given in \eqref{Z3}, we can find $t_1,t_2>0$ such that (see the proof of Lemma \ref{Absorbing})
\begin{align*}
	-\frac{\sigma^2}{2}t+2\sigma\int_{0}^{t} y(\theta_{r}\omega)\d r+2\sigma y(\theta_{t}\omega)\leq -\delta_1 t, \ \ \ \text{ for all } \ \ \ t\geq t_1,
\end{align*}
and 
\begin{align*}
	-\frac{\sigma^2}{2}t+2\sigma\int_{0}^{t} y(\theta_{r}\omega)\d r+\sigma y(\theta_{t}\omega)\leq -\delta_2t, \ \ \ \text{ for all } \ \ \ t\geq t_2.
\end{align*}
Taking $t_0=\max\{t_1,t_2\}$, we have for $t\geq t_0$,
\begin{align*}
	L_1(t,\omega)&\leq 2 e^{-\delta_1t}\bigg[\int_{t_0}^{t}e^{\delta_1\xi}	\big\{\|\f(\xi)\|^2_{\L^2(\mathcal{O})}+\|\nabla\land\boldsymbol{f}(\xi)\|_{\mathrm{L}^2(\mathcal{O})}^2\big\}\d\xi\nonumber\\&\qquad+\int_{-t_0}^{t_0}e^{\frac{\sigma^2}{2}\xi+2\sigma\int_{\xi}^{0} y(\theta_{r}\omega)\d r-2\sigma y(\theta_{\xi}\omega)}
	\big\{\|\f(\xi)\|^2_{\L^2(\mathcal{O})}+\|\nabla\land\boldsymbol{f}(\xi)\|_{\mathrm{L}^2(\mathcal{O})}^2\big\}\d\xi\nonumber\\&\qquad +\int_{-\infty}^{-t_0}e^{\delta_1\xi}
	\big\{\|\f(\xi)\|^2_{\L^2(\mathcal{O})}+\|\nabla\land\boldsymbol{f}(\xi)\|_{\mathrm{L}^2(\mathcal{O})}^2\big\}\d\xi\bigg]\nonumber\\&\leq  \bigg[4\int_{-\infty}^{0}e^{\delta_1\xi}	\big\{\|\f(\xi+t)\|^2_{\L^2(\mathcal{O})}+\|\nabla\land\boldsymbol{f}(\xi+t)\|_{\mathrm{L}^2(\mathcal{O})}^2\big\}\d\xi\nonumber\\&\qquad+2e^{-\delta_1t}\int_{-t_0}^{t_0}e^{\frac{\sigma^2}{2}\xi+2\sigma\int_{\xi}^{0} y(\theta_{r}\omega)\d r-2\sigma y(\theta_{\xi}\omega)}
	\big\{\|\f(\xi)\|^2_{\L^2(\mathcal{O})}+\|\nabla\land\boldsymbol{f}(\xi)\|_{\mathrm{L}^2(\mathcal{O})}^2\big\}\d\xi\bigg],
\end{align*}
and
\begin{align*}
	L_2(t,\omega)&\leq \bigg[4\int_{-\infty}^{0}e^{\delta_2\xi}	\|\nabla\land\boldsymbol{f}(\xi+t)\|_{\mathrm{L}^{\infty}(\mathcal{O})}\d\xi\nonumber\\&\qquad+2e^{-\delta_2t}\int_{-t_0}^{t_0}e^{\frac{\sigma^2}{2}\xi+\sigma\int_{\xi}^{0} y(\theta_{r}\omega)\d r-2\sigma y(\theta_{\xi}\omega)}
	\|\nabla\land\boldsymbol{f}(\xi)\|_{\mathrm{L}^{\infty}(\mathcal{O})}\d\xi\bigg].
\end{align*}
%where $R$ is the same as obtained in the proof of Lemma \ref{Absorbing}. Hypotheses \ref{Hyp-f} and
We infer from Hypothesis \ref{Hypo-f-3} that $\sup\limits_{t\geq t_0}L_1(t,\omega)$ and $\sup\limits_{t\geq t_0}L_2(t,\omega)$ are bounded. Hence, $\bigcup\limits_{t\geq t_0}\mathcal{B}(t, \omega)$ is a $\d_{\mathcal{W}}$-bounded set. Since $\V$ is compactly embedded in $\H$, then arguing similarly as in the proof of Lemma \ref{(i)}, one can conclude that $\mathcal{B}_{t_0}(\omega):=\overline{\bigcup\limits_{t\geq t_0}\mathcal{B}(t, \omega)}^{\delta}$ is a $\d_{\mathcal{W}}$-bounded set. 
\end{proof}

Now, we are ready to reveal the main result of this section, that is, weak asymptotic autonomy of pullback stochastic weak attractors.

\begin{theorem}\label{WAA-EE}
	Suppose that Hypotheses \ref{Hypo_f-N} and \ref{Hypo-f-3}  are satisfied. Then the $\d_{\mathcal{W}}$-bounded and $\delta$-compact pullback stochastic weak attractor $\mathscr{A}(t,\omega)$ of the system \eqref{1} and stochastic weak attractor $\mathscr{A}_{\infty}(\omega)$ of the system \eqref{11} satisfy the \textbf{weak asymptotic autonomy}, that is, 
	 \begin{align*}
		\lim_{t\to+\infty}\mathrm{dist}_{\mathcal{W}}^{\delta}(\mathscr{A}(t,\omega),\mathscr{A}_{\infty}(\omega))=0, \ \ \mathbb{P}\text{-a.e. } \omega\in\Omega,
	\end{align*}
	where $\mathrm{dist}_{\mathcal{W}}^{\delta}(\cdot,\cdot)$ denotes the Hausdorff semi-distance between two non-empty subsets of some Banach space $\mathcal{W}$ with respect to the metric $\delta$, that is, for non-empty sets $A,B\subset \mathcal{W}$ $$\mathrm{dist}_{\mathcal{W}}^{\delta}(A,B)=\sup_{a\in A}\inf_{b\in B} \delta(a,b).$$
\end{theorem}
\begin{proof}
		Lemma \ref{(i)}, Proposition \ref{For_conver-N} and Lemma \ref{(iii)} demonstrate that conditions $\mathrm{(i)}-\mathrm{(iii)}$ of Theorem \ref{WAA-MT} are satisfied. Hence, an application of Theorem \ref{WAA-MT} completes the proof.
\end{proof}
\iffalse 
\begin{remark}
	\textcolor{red}{Euler equations with linear delay, Euler equations with hereditary viscosity}
\end{remark}

\begin{remark}
	\textcolor{red}{Give a remark on upper semicontinuity!}
\end{remark}
\fi 
\vskip 2mm
\noindent

	\medskip\noindent
	{\bf Acknowledgments:} The first author would like to thank the Council of Scientific $\&$ Industrial Research (CSIR), India for financial assistance (File No. 09/143(0938)/2019-EMR-I).  M. T. Mohan would  like to thank the Department of Science and Technology (DST), Govt of India for Innovation in Science Pursuit for Inspired Research (INSPIRE) Faculty Award (IFA17-MA110).

\end{document}